\documentclass[10 pt]{amsart}

\usepackage{amssymb,amsmath,amsfonts,mathrsfs}
\usepackage[title]{appendix}

\newtheorem{theorem}{Theorem}
\newtheorem{lemma}{Lemma}
\newtheorem{corollary}{Corollary}
\newtheorem{prop}{Proposition}

\theoremstyle{definition}
\newtheorem{definition}{Definition}

\theoremstyle{remark}
\newtheorem{remark}{Remark}

\numberwithin{equation}{section}

\newcommand{\rcomp}{C_{c}^{\infty}(\mathbb{R}^n)}

\newcommand{\dom}{\operatorname{Dom}}

\newcommand{\supp}{\operatorname{supp}}

\newcommand\RR{\mathbb{R}}

\newcommand\NN{\mathbb{N}}

\newcommand\p{\partial}

\newcommand\fpi{F\Psi^{-\infty}}
\newcommand\fprl{F\Psi}
\newcommand\uprl{U\Psi}

\newcommand\ueprl{UE\Psi}

\newcommand{\xcomp}{C_{c}^{\infty}(X)}

\begin{document}

\title[Extended Sobolev scale]{Extended Sobolev Scale on Non-Compact Manifolds}
\author{Ognjen Milatovic}
\address{Department of Mathematics and Statistics\\
         University of North Florida   \\
       Jacksonville, FL 32224 \\
        USA
           }
\email{omilatov@unf.edu}



\begin{abstract}
Adapting the definition of ``extended Sobolev scale" on compact manifolds by Mikhailets and Murach to the setting of a (generally non-compact) manifold of bounded geometry $X$, we define the ``extended Sobolev scale" $H^{\varphi}(X)$, where $\varphi$ is a function which is $RO$-varying at infinity. With the help of the scale $H^{\varphi}(X)$, we obtain a description of all Hilbert function-spaces that serve as interpolation spaces with respect to a pair of Sobolev spaces $[H^{(s_0)}(X), H^{(s_1)}(X)]$, with $s_0<s_1$. We use this interpolation property to establish a mapping property of proper uniform pseudo-differential operators (PUPDOs) in the context of the scale $H^{\varphi}(X)$. Additionally, using a first-order positive-definite PUPDO $A$ of elliptic type we define the ``extended $A$-scale" $H^{\varphi}_{A}(X)$ and show that it coincides, up to norm equivalence, with the scale $H^{\varphi}(X)$. Besides the mentioned results, we show that further properties of the $H^{\varphi}$-scale, originally established by Mikhailets and Murach on $\RR^n$ and on compact manifolds, carry over to manifolds of bounded geometry.
\end{abstract}

\maketitle

\section{Introduction}\label{S:S-intro}
Ever since their debut in the 1930s, Sobolev spaces have played an important role in analysis and PDE. In particular, interpolation qualities of the Sobolev scale make it possible to extend various properties of integer-order Soblev spaces to real-order Sobolev spaces. However, as pointed out in~\cite{MM-14}, the usual Sobolev scale $\{H^{(s)}(\RR^n)\colon s\in \RR\}$ is not sufficiently finely calibrated to tackle some mathematical problems. To overcome this difficulty, in their seminal papers~\cite{MM-09,MM-13} (see also Section 2.4 of their monograph~\cite{MM-14}), Mikhailets and Murach proposed the so-called extended Sobolev scale $H^{\varphi}(\RR^n)$, defined similar to $H^{(s)}(\RR^n)$ but with $\varphi(\langle\xi\rangle)$ in place of $\langle\xi\rangle^{s}$, where $\varphi$ is a function $RO$-varying at $\infty$ and satisfies some additional properties; see Section~\ref{SS-1-8} below for precise description of the function class $RO$. It is worth mentioning that the scale $H^{\varphi}(\RR^n)$ is a generalization of the so-called refined Sobolev scale, introduced by Mikhailets and Murach earlier in~\cite{MM-07} (see also Section 1.3 in~\cite{MM-14}). What makes the scale $H^{\varphi}(\RR^n)$ particularly interesting is the following interpolation property (established in~\cite{MM-13}; see also~\cite{MM-15} for bounded domains $\Omega\subset\RR^n$ with Lipschitz boundary): a Hilbert space $\mathscr{S}$ is an interpolation space with respect to a pair (see Section~\ref{SS-2-1} below for this concept) of the form
\begin{equation*}
[H^{(s_0)}(\RR^n),H^{(s_1)}(\RR^n)], \qquad -\infty<s_0<s_1<\infty,
\end{equation*}
if and only if $\mathscr{S}=H^{\varphi}(\RR^n)$, for some $\varphi\in RO$.

Owing to the mentioned interpolation property (and other useful attributes studied in~\cite{MM-07,MM-13, MM-15}), refined (and extended) Sobolev spaces (as isotropic cases  of the so-called H\"ormander spaces), provide a convenient setting for building the theory of elliptic boundary-value problems on $\RR^n$ and for obtaining various results from spectral theory of differential operators on $\RR^n$ in analogy to those holding in the realm of the usual Sobolev spaces.  As a testimony of fruitful activity in this direction during the last fifteen years, besides the papers mentioned so far, we refer the reader to the monograph~\cite{MM-14}, the papers~\cite{DMM-21,AK-16,MM-21,MZ-24}, the survey~\cite{MMC-25}, and numerous references therein. For theory of parabolic boundary-value problems on $\RR^n$ in the so-called anisotropic generalized Sobolev spaces, see the paper~\cite{LMM-21-paper}, the monograph~\cite{LMM-21}, and references therein.

In addition to the properties of the $H^{\varphi}$-scale on $\RR^n$, Mikhailets, Murach, and their collaborators have studied $H^{\varphi}$-spaces on closed manifolds (that is, compact manifolds without boundary) as well as vector bundles over closed manifolds; see the papers~\cite{MM-09,MZ-24,Z-17} and the monograph~\cite{MM-14}. (The paper~\cite{Z-17} considers the ``refined Sobolev scale" on vector bundles over closed manifolds, a special case of the $H^{\varphi}$-scale.) In particular, an analogue of the aforementioned interpolation property (and its counterparts), along with the mapping property with respect to pseudo-differential operators with symbols in the H\"ormander-type class on closed manifolds, were established in Section 2.4 of~\cite{MM-14}. More recently, $H^{\varphi}$-scale on compact manifolds with boundary was investigated in~\cite{Kas-19}.

In the setting of a (not necessarily compact) manifold of bounded geometry $X$ (see Section~\ref{SS-1-3} for the meaning of ``bounded geometry"), the author of~\cite{T-86} studied Besov spaces and Triebel--Lizorkin spaces (which include the usual Sobolev spaces $H^{(s)}(X)$, $s\in\RR$). Modifying the definition of the Sobolev scale $H^{(s)}(X)$, where $X$ is a manifold of bounded geometry, in Section~\ref{SS-1-10} we define the (extended) Sobolev scale $H^{\varphi}(X)$, with $\varphi\in RO$. Later, in Theorem~\ref{T:ind-triv}, we show that the definition of $H^{\varphi}(X)$ is independent (up to norm equivalence) of the geodesic trivialization.

Important examples of manifolds of bounded geometry are Lie groups (or more general homogeneous manifolds with invariant metrics) and covering manifolds of compact manifolds (with a Riemannian metric lifted from the base manifold). The main goal of this article is to show that the aforementioned interpolation property (and its variants) hold in the setting of (extended) Sobolev spaces on manifolds of bounded geometry; see Theorem~\ref{T:main-1} and Theorem~\ref{T:main-2} below. We do this via the key Proposition~\ref{L-5}, where the bounded geometry features discussed in~\cite{G-13,Kor-91,T-86} allow us to apply a suitable localization technique and use the corresponding interpolation properties for $H^{\varphi}(\RR^n)$.

Furthermore, in analogy with a ``quadratic interpolation" result of~\cite{MM-15} for bounded domains $\Omega\subset\RR^n$ with Lipschitz boundary, we show that (see Theorem~\ref{T:main-3} below) the class $H^{\varphi}(X)$ is closed under interpolation with a function parameter. In Theorem~\ref{T:main-4} and Theorem~\ref{T:main-6} we establish further properties of the scale $H^{\varphi}(X)$, including the density of $\xcomp$ (smooth compactly supported functions) in $H^{\varphi}(X)$ and an embedding result $H^{\varphi}(X)\hookrightarrow C^{k}_{b}(X)$, where $C^{k}_{b}(X)$ stands for $C^k$-bounded functions on $X$.

According to the theory of the proper uniform pseudo-differential operators (PUPDOs) on manifolds of bounded geometry (see Section~\ref{SS-1-18} below for details), developed in~\cite{MS-1} for unimodular Lie groups and subsequently generalized in~\cite{Kor-91} to manifolds of bounded geometry, a PUPDO $A$ of order $m\in\RR$ on a manifold of bounded geometry $X$ extends to a bounded linear operator $A\colon H^{(s)}(X)\to H^{(s-m)}(X)$. With this in mind, we use the interpolation result of Theorem~\ref{T:main-2} to prove the corresponding mapping property for PUPDOs in the scale $H^{\varphi}(X)$, $\varphi\in RO$; see Theorem~\ref{T:main-5} below. For mapping properties of pseudo-differential operators in the context of Besov spaces and Triebel--Lizorkin spaces on manifolds of bounded geometry, see~\cite{Sk-98}. For further development of the uniform localization technique and its applications to the spectral theory of elliptic operators, see~\cite{T-99}.

As in the corresponding definition in~\cite{MM-21} for closed manifolds, working in the setting of  a  manifold of bounded geometry $X$,  for $\varphi\in RO$ and for a first-order PUPDO $A$ of elliptic type (see Definition~\ref{D-6-x} below) satisfying $(Au,u)\geq \|u\|^2$ for all $u\in\xcomp$, we define the so-called extended Hilbert $A$-scale $H_{A}^{\varphi}(X)$. With this background, in Theorem~\ref{T:main-7} we show that, up to norm equivalence, we have $H_{A}^{\varphi}(X)=H^{\varphi}(X)$.

Lastly, we remark that, similar to the extended Sobolev scale in the context of vector bundles over closed manifolds in~\cite{Z-17},  all results of our article hold for the extended Sobolev scale $H^{\varphi}(E)$, $\varphi\in RO$, where $E$  is a Hermitian vector bundle of bounded geometry (over a manifold of bounded geometry). (The scale $H^{\varphi}(E)$, $\varphi\in RO$, can be defined by the modifying the corresponding definition of the scale $H^{(s)}(E)$, $s\in\RR$, from~\cite{G-13}.) In particular, the mapping property (Theorem~\ref{T:main-5}) can be formulated for PUPDOs on manifolds of bounded geometry acting on sections of Hermitian vector bundles of bounded geometry.


The article consists of twelve sections. In Section~\ref{S:S-1} we summarize the basic notations, describe relevant geometric concepts (bounded geometry, geodesic coordinates, geodesic trivialization, and a special partition of unity). Furthermore, Section~\ref{S:S-1} contains the definitions of the usual Sobolev scale, $RO$-varying functions, the extended Sobolev scale, PUPDOs on manifolds of bounded geometry, and basic concepts from interpolation theory of Hilbert spaces. Additionally, in Section~\ref{S:S-1} (more specifically subsections~\ref{SS-1-12}--\ref{SS-1-15}, \ref{SS-1-19}--\ref{SS-1-20}, and \ref{SS-1-22}) we state the main results (eight theorems and one corollary) of the article. Section~\ref{S:S-2} contains the statements of preliminary facts from interpolation theory, including basic properties of the scale $H^{\varphi}(\RR^n)$. Sections~\ref{S:S-3}--\ref{S:S-10} contain the proofs of the main results.

\section{Geometric Setting and Main Results}\label{S:S-1}
The goal of this section is to recall the concept of geodesic coordinates on a Riemannian manifold, describe a manifold of bounded geometry and the corresponding geodesic trivialization coming from a special type of partition of unity. For more details on these concepts, we refer the reader to the papers~\cite{G-13, Kor-91,Sh-92, T-86}.

\subsection{Partial Derivative Notations}
By an $n$-dimensional multiindex $\alpha$ we mean an element of $\NN_0^{n}$, where $\NN_0:=\{0,1,2,\dots\}$. For a multi-index $\alpha=(\alpha_1, \alpha_2, \dots, \alpha_n)$ and $x\in \RR^n$, we define $|\alpha|:=\alpha_1+ \alpha_2+ \dots+ \alpha_n$, $\alpha!:=\alpha_1\alpha_2\dots\alpha_n$, and $x^{\alpha}:=x_1^{\alpha_1}x_2^{\alpha_2}\dots x_n^{\alpha_n}$. For a complex-valued function $u(x_1,x_2,\dots, x_n)$ and $\alpha\in\NN_0^n$, the partial derivative notations $\p^{\alpha}u$ and $D^{\alpha}$ have the following meaning:
\begin{equation}\label{E:pdv-not}
  \p^{\alpha}u:=\frac{\p^{|\alpha|}u}{\p {x_1}^{\alpha_1}\p {x_2}^{\alpha_2}\dots \p {x_n}^{\alpha_n}}, \quad  D^{\alpha}u:=(-i)^{|\alpha|}\p^{\alpha}u,
\end{equation}
where $i$ is the imaginary unit. When talking about functions $a(x,\xi)$, where $x\in\RR^n$ and $\xi\in\RR^n$, the expression $\p_{\xi}^{\alpha}\p_{x}^{\beta}a(x,\xi)$ denotes the partial derivatives with respect to $\xi$ and $x$ of orders $|\alpha|$ and $|\beta|$ respectively.

For $\xi=(\xi_1,\xi_2,\dots, \xi_n)\in\RR^n$ we set
\begin{equation}\label{E:jb-1}
\langle\xi\rangle:=(1+\xi^2_1+\xi^2_2+\dots+ \xi^2_n)^{1/2}.
\end{equation}

\subsection{Some Geometric Terminology and Notations}\label{SS-1-1}
In this article, $X$ is an $n$-dimensional connected Riemannian manifold without boundary equipped with a metric $g$.
The symbol $\mu_{g}$ stands for the volume element of $X$ corresponding to $g$. Denoting by $d_{g}(\cdot,\cdot)$ the distance function with respect to $g$, for $x_0\in X$ and $r>0$ we define
\begin{equation*}
  B(x_0,r):=\{x\in X\colon d_{g}(x,x_0)<r\}.
\end{equation*}
The symbols $T_{x}X$ and $T^*_{x}X$ denote the tangent and cotangent space at $x\in X$ respectively. The notation $O(X)$ indicates the orthonormal frame  bundle of $X$ with the projection $\pi\colon O(X)\to X$. Here, the wording ``frame $e\in O(X)$ with $\pi(e)=x$" means that $e$ is an isometric isomorphism $e\colon \RR^n\to T_{x}X$.

\subsection{Geodesic Coordinates}\label{SS-1-2} Let $K_{\tau}$ be an open ball with radius $\tau>0$ centered at the origin of $T_{x}X$. It is known that there exists $r>0$ such that the exponential map $\textrm{exp}_{x}\colon K_{r}\to B(x,r)$ is a diffeomorphism. By \emph{injectivity radius of $X$ at $x$} we mean the supremum $r_{x}$ of such numbers $r>0$. Moreover, by \emph{injectivity radius of $X$} we mean $r_{inj}:=\inf_{x\in X}r_{x}$. Fixing $x\in X$ and denoting by $\widehat{K}_{r}$ the open ball of radius $r<r_{x}$ centered at $0\in\RR^n$, a frame $e\in O(X)$ with $\pi(e)=x$ leads to a diffeomorphism $\gamma_{e}\colon \widehat{K}_{r}\to B(x,r)$ by means of the following composition:
\begin{equation}\label{E:geo-def-1}
\gamma_{e}=\textrm{exp}_{x}\circ e
\end{equation}
The mapping $\gamma_{e}$ in~(\ref{E:geo-def-1}) is called \emph{a geodesic coordinate system}. (Some authors use the term \emph{normal coordinate system}.)
Additionally, we have a diffeomorphism
\begin{equation}\label{E:geo-def-kappa}
\kappa_{e}:=\gamma_{e}^{-1}\colon B(x,r)\to \widehat{K}_{r}.
\end{equation}
Let $0<\varepsilon<r_{inj}$ and let $\{x_j\}_{j=1}^{\infty}$ be a set of points such that $\{V_{j}:=B(x_{j},\varepsilon)\}_{j=1}^{\infty}$ covers $X$. Furthermore, let $\{\gamma_{e_j}\}_{j=1}^{\infty}$ be as in~(\ref{E:geo-def-1}) with $x=x_j$. Then, as indicated in example 2.3 of~\cite{G-13}, the collection $(V_{j},\gamma_{e_j})_{j=1}^{\infty}$ is an atlas on $X$--often called (in literature) a \emph{geodesic atlas}.

\subsection{Manifold of Bounded Geometry}\label{SS-1-3} For the remainder of this section we will work in the setting of a \emph{manifold bounded geometry}, by which we mean a Riemannian manifold satisfying the following two conditions:
\begin{enumerate}
  \item [(i)] $r_{inj}>0$;
  \item [(ii)] all covariant derivatives of the Riemann curvature tensor are bounded on $X$.
\end{enumerate}

There is an equivalent characterization of ``bounded geometry" (see Proposition 1.2 in~\cite{Kor-91}):

\medskip

We say that a manifold $X$ has \emph{bounded geometry} if there exists a ball $B\subset \RR^n$ centered at $0$ such that
\begin{enumerate}
  \item [(i)] for each frame $e\in O(X)$, the corresponding geodesic coordinate system $\gamma_{e}$ is defined on  $B$;
  \item [(ii)] in the system $\gamma_{e}\colon B\to X$, the Christoffel symbols of the Levi--Civita connection have derivatives of all orders, and the corresponding derivatives are bounded on $B$ uniformly with respect to the frames $e\in O(X)$.
\end{enumerate}

\begin{remark}\label{R:b-geo} Here we record an important observation: on a manifold of bounded geometry $X$, a frame $e\in O(X)$ gives rise to a coordinate system $\gamma_{e}\colon B\to X$, whereby the same ball $B\subset \RR^n$ is used for all frames $e$. This system $\gamma_{e}$ is referred to as \emph{the standard coordinate system} corresponding to $e\in O(X)$. In what follows, we denote by $\tau_0$ the radius of the ball $B$. In particular, if $e\in O(X)$ is a frame such that $\pi(e)=x$, we have a diffeomorphism $\gamma_{e}\colon  B\to B(x,\tau_0)$.
\end{remark}

We give another equivalent characterization of ``bounded geometry" (see Proposition 1.3 in~\cite{Kor-91} or Remark 3.12(ii) of~\cite{G-13}):

\medskip

Keeping in mind the notations of Section~\ref{SS-1-2} above, we say that a manifold $X$ has \emph{bounded geometry} if the following two conditions are satisfied:
\begin{enumerate}
  \item [(i)] $r_{inj}>0$;
  \item [(ii)] Let $(V_{j},\gamma_{e_j})_{j=1}^{\infty}$ be an arbitrary geodesic atlas of $X$. Then, for all $k\in\NN_{0}$ there exists a constant $C_k$, such that for all $i,j\in\NN$ with $V_i\cap V_j\neq\emptyset$ we have
  \begin{equation}\label{E:geo-atlas-i-j}
  |\p^{\alpha}[((\gamma_{e_i})^{-1}\circ \gamma_{e_j})(x)]|\leq C_{k},
  \end{equation}
for all $x\in \widehat{K}_{\varepsilon}$ and all $\alpha\in{\NN}_{0}^{n}$ such that $|\alpha|\leq k$. (Here, the notation $\p^{\alpha}$ is as in~(\ref{E:pdv-not}) and $\widehat{K}_{\varepsilon}$ is as in Section~\ref{SS-1-2}.)
\end{enumerate}
\begin{remark}\label{R:geo-atlas-i-j} We stress, for future reference, that the constant $C_{k}$ in~(\ref{E:geo-atlas-i-j}) does not depend on the indices $i$ and $j$ of the functions $\gamma_{e_i}$ and $\gamma_{e_j}$.
\end{remark}

\subsection{Special Covering, Partition of Unity}\label{SS-1-4} As shown by the author of~\cite{Kor-91} (see Lemma 2.3 and Lemma 2.4 there), a manifold of bounded geometry has the following properties:

\begin{itemize}
  \item [(C1)] For all $0<\varepsilon<r_{inj}$ there exists a geodesic atlas $(V_j=B(x_j,\varepsilon), \gamma_{e_j})_{j=1}^{\infty}$, with $\gamma_{e_j}$ as in~(\ref{E:geo-def-1}), such that the covering $\mathscr{C}:=\{B(x_j,\varepsilon)\}_{j=1}^{\infty}$ of $X$ has a \emph{finite order} $N_0\in\NN$. (The latter term means that each set $B(x_j,\varepsilon)$ from the cover $\mathscr{C}$ is intersected by at most $N_0$ members of $\mathscr{C}$.)

  \item [(C2)] In the notations of (C1), there exists a smooth partition of unity $\{h_j\}_{j=1}^{\infty}$ on $X$, with $\sum_{j=1}^{\infty}h_j=1$, such that
\begin{enumerate}
  \item [(i)] the collection $\{h_j\}_{j=1}^{\infty}$ is subordinate to the covering $\{B(x_j,\varepsilon)\}_{j=1}^{\infty}$ from (C1), that is $\supp h_{j}\subset V_j$;
  \item [(ii)] for all $k\in\NN_0^{n}$ there exists a constant $C_{k}>0$ such that
  \begin{equation}\label{E:geodesic-atlas-h}
   |\p^{\alpha}[(h_j\circ\gamma_{e_j})(x)]|\leq C_{k},
  \end{equation}
for all $x\in \widehat{K}_{\varepsilon}$, all $\alpha\in\NN_{0}^{n}$ with $|\alpha|\leq k$,  and all $j\in\NN$. (Here, the notation $\p^{\alpha}$ is as in~(\ref{E:pdv-not}) and $\widehat{K}_{\varepsilon}$ is as in Section~\ref{SS-1-2}.)
\end{enumerate}

\end{itemize}
\begin{remark}\label{R:geo-atlas-h} We stress, for future reference, that the constant $C_{k}$ in~(\ref{E:geodesic-atlas-h}) does not depend on the the index $j$ corresponding to $h_j$ and $e_j$.
\end{remark}

\subsection{Geodesic Trivialization}\label{SS-1-5} Keeping in mind the notations of Section~\ref{SS-1-4}, the triplet $(V_j, \gamma_{e_j}, h_j)_{j=1}^{\infty}$ (or the triplet $(V_j, \kappa_{e_j}, h_j)_{j=1}^{\infty}$ with $\kappa_{e_j}$ as in~(\ref{E:geo-def-kappa})) is called a \emph{geodesic trivialization}.

\subsection{Function Spaces}\label{SS-1-6}
The symbol $C^{\infty}(X)$ denotes the space of smooth complex-valued functions on $X$, the notation $\xcomp$ stands for compactly supported elements of $C^{\infty}(X)$, and  $\mathcal{D}'(X)$ indicates the space of distributions on $X$. (Here, the word ``distributions" is meant as continuous linear functionals on $\xcomp$.)  For $k=0,1,\dots,\infty$, the space $C^{k}_{b}(X)$ of $C^{k}$-bounded functions on $X$ consists of $u\in C^{k}(X)$ such that the family $\{u\circ \gamma_{e}\colon e\in O(X)\}$ is bounded in the (Fr\'echet) space $C^{k}(\overline{B})$, where $\gamma_{e}\colon B\to X$ is as in Remark~\ref{R:b-geo}. The space of complex-valued square integrable functions on $X$ will be denoted by $L^2(X)$. We will use $\|\cdot\|$ to indicate the norm in $L^2(X)$ corresponding to the inner product
\begin{equation}\label{E:inner-l-2-x}
  (u,v):=\int_{X} u(x)\overline{v(x)}\,d\mu_{g}(x).
\end{equation}

In the case $X=\RR^n$, the function spaces mentioned so far will be denoted by $C^{\infty}(\RR^n)$, $\rcomp$, $C^{k}_{b}(\RR^n)$, $\mathcal{D}'(\RR^n)$ and $L^2(\RR^n)$. Furthermore, $S(\RR^n)$ and $S'(\RR^n)$ denote the Schwartz space (that is, the space of rapidly decreasing functions) and the space of tempered distributions respectively.  The notation $\widehat{u}$ or $\mathcal{F}u$ indicates the Fourier transform of a function $u\in S(\RR^n)$:
\begin{equation}\label{E:FT}
\widehat{u}(\xi):=\int_{{\RR}^{n}} e^{-iy\cdot\xi}u(y)\,dy.
\end{equation}
It is well known that $\mathcal{F}\colon S(\RR^n)\to  S(\RR^n)$ is a topological linear isomorphism. The Fourier transform of a distribution $u\in S'(\RR^n)$ will also be denoted by $\widehat{u}$ or $\mathcal{F}u$.

Having listed these notations, we are ready to recall the definition of the ($L^2$-type) Sobolev scale on $\RR^n$.

\subsection{Sobolev Scale on $\RR^n$}\label{SS-1-7} For $s\in\RR$, $H^{(s)}(\RR^n)$  is defined as
\begin{equation}\label{E:sob-r-1}
H^{(s)}(\RR^n):=\{u\in S'(\RR^n)\colon \langle\xi\rangle^{s}\widehat{u}\in L^2(\RR^n)\}.
\end{equation}
It turns out that $H^{(s)}(\RR^n)$ is a Hilbert space with the inner product
\begin{equation}\label{E:sob-r-2}
(u,v)_{H^{(s)}}:=\int_{\RR^n}\langle\xi\rangle ^{2s}u(x)\overline{v(x)}\,dx,
\end{equation}
and the norm corresponding to~(\ref{E:sob-r-2}) will be denoted by $\|\cdot\|_{H^{(s)}}$.

Before defining the ``extended Sobolev scale" in the sense of~\cite{MM-09, MM-13, MM-14}, we describe the so-called $RO$-varying functions.

\subsection{$RO$-varying Functions at $\infty$}\label{SS-1-8} In the definitions of this section we follow the terminology of Section 2.4.1 in~\cite{MM-14}.
A function $\varphi\colon [1,\infty)\to (0,\infty)$ is said to be \emph{$RO$-varying at infinity} if
\begin{enumerate}
  \item [(i)] $\varphi$ is Borel measurable
  \item [(ii)] there exist numbers $a>1$ and $c\geq 1$ (depending on $\varphi$) such that
  \begin{equation}\label{E:RO-1}
    c^{-1}\leq\frac{\varphi(\lambda t)}{\varphi(t)}\leq c, \qquad\textrm{for all }t\geq 1,\,\,\lambda\in [1,a].
  \end{equation}
\end{enumerate}
In the subsequent discussion, the relation $\varphi\in RO$ means that a function $\varphi\colon [1,\infty)\to (0,\infty)$ is $RO$-varying at infinity.

By Proposition 1 of~\cite{MM-13}, if $\varphi\in RO$, then $\varphi$ is bounded and separated from zero on every interval of the form $[1,b]$ with $b>1$. Furthermore, according to the same proposition, the condition~(\ref{E:RO-1}) has the following equivalent formulation: there exist numbers $s_0\leq s_1$ and $c\geq 1$  such that
\begin{equation}\label{E:RO-2}
   t^{-s_0}\varphi(t)\leq c\tau^{-s_0}\varphi(\tau),\qquad \tau^{-s_1}\varphi(\tau)\leq ct^{-s_1}\varphi(t) \qquad\textrm{for all }1\leq t\leq \tau.
  \end{equation}

Lastly, we recall the concept of lower/upper Matuszewska indices of $\varphi\in RO$. Let $\varphi\in RO$. Setting $\lambda:=\frac{\tau}{t}$, we can rewrite~(\ref{E:RO-2}) as
\begin{equation}\label{E:RO-3}
    c^{-1}\lambda^{s_0}\leq\frac{\varphi(\lambda t)}{\varphi(t)}\leq c\lambda^{s_1}, \qquad\textrm{for all }t\geq 1,\,\,\lambda\geq 1.
  \end{equation}
We define the \emph{lower Matuszewska index} $\sigma_0(\varphi)$ as the supremum of all $s_0\in\RR$ such that the leftmost inequality in~(\ref{E:RO-3}) is satisfied. Likewise, we define the \emph{upper Matuszewska index} $\sigma_1(\varphi)$ as the infimum of all $s_1\in\RR$ such that the rightmost inequality in~(\ref{E:RO-3}) is satisfied. Note that $-\infty<\sigma_0(\varphi)\leq \sigma_1(\varphi)<\infty$.

\subsection{Extended Sobolev Scale on $\RR^n$}\label{SS-1-9} For $\varphi\in RO$ we define $H^{\varphi}(\RR^n)$ as follows:
\begin{equation}\label{E:sob-phi-1}
H^{\varphi}(\RR^n):=\{u\in S'(\RR^n)\colon \varphi(\langle\xi\rangle)\widehat{u}\in L^2(\RR^n)\}.
\end{equation}
As in the case of the usual Sobolev scale,  $H^{\varphi}(\RR^n)$ is a Hilbert space with the inner product
\begin{equation}\label{E:sob-phi-2}
(u,v)_{H^{\varphi}(\RR^n)}:=\int_{\RR^n} [\varphi(\langle\xi\rangle)]^2 u(x)\overline{v(x)}\,dx,
\end{equation}
and the norm corresponding to~(\ref{E:sob-phi-2}) will be denoted by $\|\cdot\|_{H^{\varphi}(\RR^n)}$.
The space $H^{\varphi}(\RR^n)$ is a special case of the so-called H\"ormander space from Section 10.1 of~\cite{Hor-pdo-book}. Note that if $\varphi(t)=t^{s}$, $s\in\RR$, the space $H^{\varphi}(\RR^n)$ leads to the usual Sobolev space $H^{(s)}(\RR^n)$. In this article, the term \emph{extended Sobolev scale} on $\RR^n$ refers to the class of spaces $\{H^{\varphi}(\RR^n)\colon \varphi\in \RR\}$.

Having defined the (extended) Sobolev scale on $\RR^n$ we can now use the geodesic trivialization from Section~\ref{SS-1-5} above to describe the corresponding (extended) Sobolev scale on $X$.

\subsection{Extended Sobolev Scale on $X$}\label{SS-1-10} Let $\mathscr{T}=(V_j, \gamma_{e_j}, h_j)_{j=1}^{\infty}$ be a geodesic trivialization as in Section~\ref{SS-1-5}. For $\varphi\in RO$, we define $H^{\varphi}(X,\mathscr{T})$ as the space of all $u\in \mathcal{D}'(X)$ such that
\begin{equation}\label{E:sob-phi-norm-X}
\|u\|^2_{H^{\varphi}(X,\mathscr{T})}:=\sum_{j=1}^{\infty}\|(h_ju)\circ \gamma_{e_j})|\|^2_{H^{\varphi}(\RR^n)}<\infty,
\end{equation}
where $\|\cdot\|_{H^{\varphi}(\RR^n)}$ is as in Section~\ref{SS-1-7}. The right hand side of~(\ref{E:sob-phi-norm-X}) suggests that we view $(h_ju)\circ \gamma_{e_j}$ as a function defined on all of $\RR^n$. To explain this, we refer back to the Section~\ref{SS-1-5}, where $\gamma_{e_j}\colon\widehat{K}_{\varepsilon}\to V_j$ is a deiffeomorphism, as described in~(\ref{E:geo-def-1}). This, together with the inclusion $\supp h_j\subset V_j$, tells us that we can view $(h_ju)\circ \gamma_{e_j}$ as a function defined on all of $\RR^n$, extended by zero on the set $\RR^n\backslash \widehat{K}_{\varepsilon}$.

In the special case $\varphi(t)=t^{s}$, $s\in\RR$, the space $H^{\varphi}(X,\mathscr{T})$ leads to the usual Sobolev space $H^{(s)}(X,\mathscr{T})$.

\begin{remark}\label{R:ind-triv} In Theorem~\ref{T:ind-triv} we will show that, up to the norm equivalence, the space $H^{\varphi}(X,\mathscr{T})$ does not depend on the choice of geodesic trivialization $\mathscr{T}=(V_j, \gamma_{e_j}, h_j)_{j=1}^{\infty}$. To simplify our notations, instead of $H^{\varphi}(X,\mathscr{T})$, we shall start using the symbol $H^{\varphi}(X)$ immediately, rather than waiting until we prove Theorem~\ref{T:ind-triv}.
\end{remark}


In addition to the extended Sobolev scale on $X$, the first theorem of our article uses some terminology from Section 1.1.1 and Section 1.1.2 of~\cite{MM-14} concerning interpolation of a pair of Hilbert spaces with a function parameter, which we recall below.

\subsection{Interpolation Between Hilbert Spaces}\label{SS-2-1} By an \emph{admissible pair} of separable complex Hilbert spaces we mean an ordered pair $[\mathscr{H}_{0},\mathscr{H}_{1}]$ such that $\mathscr{H}_{1}\hookrightarrow\mathscr{H}_{0}$, with the embedding being continuous and dense.  As indicated in Section 1.2.1 of~\cite{Lions-72}, every admissible pair $[\mathscr{H}_{0},\mathscr{H}_{1}]$ is equipped with a so-called \emph{generating operator} $J$, such that
\begin{enumerate}
  \item [(i)] $J$ is a positive self-adjoint operator in $\mathscr{H}_{0}$ with $\dom(J)=\mathscr{H}_{1}$;
   \item [(ii)] $\|Ju\|_{\mathscr{H}_{0}}=\|u\|_{\mathscr{H}_{1}}$, for all $u\in \dom(J)=\mathscr{H}_{1}$.
\end{enumerate}

It turns out (see Section 1.1.1 in~\cite{MM-14}) that the operator $J$ is uniquely determined by the admissible pair $[\mathscr{H}_{0},\mathscr{H}_{1}]$.

Let $\mathcal{B}$ be the set of all Borel measurable functions $\psi\colon (0,\infty)\to (0,\infty)$ satisfying the following two properties: $\psi$ is bounded that on every interval $[a,b]$ with $0 <a<b<\infty$, and $\frac{1}{\psi}$ is bounded on every interval $(c,\infty)$ with $c > 0$. For an admissible pair $\mathscr{H}:=[\mathscr{H}_{0},\mathscr{H}_{1}]$ with generating operator $J$ and for a function $\psi\in \mathcal{B}$, spectral calculus gives rise to a (positive self-adjoint) operator $\psi(J)$ in $\mathscr{H}_{0}$. We define a (separable, Hilbert) space $[\mathscr{H}_{0},\mathscr{H}_{1}]_{\psi}$ (or, in simpler notation, $\mathscr{H}_{\psi}$) as follows: $[\mathscr{H}_{0},\mathscr{H}_{1}]_{\psi}:=\dom(\psi(J))$ with the inner product
\begin{equation}\label{E:H-psi}
(u,v)_{\mathscr{H}_{\psi}}:=(\psi(J)u,\psi(J)v)_{\mathscr{H}_{0}},
\end{equation}
and the corresponding norm $\|u\|_{\mathscr{H}_{\psi}}:=\|\psi(J)u\|_{\mathscr{H}_{0}}$,
where $(\cdot,\cdot)_{\mathscr{H}_{0}}$ and $\|\cdot\|_{\mathscr{H}_{0}}$ are the inner product and the norm in ${\mathscr{H}_{0}}$.

Having defined the space $[\mathscr{H}_{0},\mathscr{H}_{1}]_{\psi}$, we can now give the definition of an interpolation parameter $\psi\in\mathcal{B}$. We say that a function $\psi\in\mathcal{B}$ is an \emph{interpolation parameter} if the following condition is satisfied for all admissible pairs $\mathscr{H}:=[\mathscr{H}_{0},\mathscr{H}_{1}]$ and $\mathscr{K }:=[\mathscr{K}_{0},\mathscr{K}_{1}]$ and for all linear operators $T$ with $\mathscr{H}_{0}\subseteq \dom(T)$: if the restrictions $T|_{\mathscr{H}_{0}}$ and $T|_{\mathscr{H}_{1}}$ act as bounded linear operators $T\colon \mathscr{H}_{0}\to \mathscr{K}_{0}$ and $T\colon \mathscr{H}_{1}\to \mathscr{K}_{1}$, then the restriction
$T|_{\mathscr{H}_{\psi}}$ acts as a bounded linear operator $T\colon \mathscr{H}_{\psi}\to \mathscr{K}_{\psi}$.

In this case, we say that the space $\mathscr{H}_{\psi}$ is obtained by interpolation of the pair $\mathscr{H}=[\mathscr{H}_{0},\mathscr{H}_{1}]$ using a function parameter $\psi\in\mathcal{B}$. Moreover, we have the following continuous dense embeddings: $\mathscr{H}_{1}\hookrightarrow\mathscr{H}_{\psi}\hookrightarrow\mathscr{H}_{0}$.

\subsection{Interpolation Space}\label{SS-1-11} Let $[\mathscr{H}_{0},\mathscr{H}_{1}]$ be an ordered pair  of separable complex Hilbert spaces such that $\mathscr{H}_{1}\hookrightarrow\mathscr{H}_{0}$, with the arrow indicating continuous embedding. We say that a Hilbert space $\mathscr{S}$ is an \emph{interpolation space with respect to a pair} $[\mathscr{H}_{0},\mathscr{H}_{1}]$ if the following conditions are satisfied:
\begin{enumerate}
  \item [(i)] we have continuous embeddings $\mathscr{H}_{1}\hookrightarrow\mathscr{S}\hookrightarrow\mathscr{H}_{0}$;
  \item [(ii)] any linear operator $T$ in $\mathscr{H}_{0}$ which acts as a bounded linear operator $T\colon \mathscr{H}_{0}\to \mathscr{H}_{0}$ and $T\colon \mathscr{H}_{1}\to \mathscr{H}_{1}$, has the property that $T\colon \mathscr{S}\to\mathscr{S}$ is also a bounded linear operator.
\end{enumerate}
\begin{remark}\label{R-rem-int} Part (ii) of the above definition leads to the following property (see Theorem 1.8 in~\cite{MM-14} or Theorem 2.4.2 in~\cite{BL-book}):
\begin{equation*}
  \|T\|_{\mathscr{S}\to \mathscr{S}}\leq C \max \{\|T\|_{\mathscr{H}_0\to \mathscr{H}_0},\|T\|_{\mathscr{H}_1\to \mathscr{H}_1}\},
\end{equation*}
where $C>0$ is a constant independent of $T$.
\end{remark}

\subsection{First Interpolation Result}\label{SS-1-12} We are ready to formulate a result about interpolation of a pair of (usual) Sobolev spaces $[H^{(s_0)}(X), H^{(s_1)}(X)]$, where $s_0<s_1$ are real numbers. By Corollary 3.8 in~\cite{Kor-91}, there is a continuous (and dense) embedding $H^{(s_1)}(X)\hookrightarrow H^{(s_0)}(X)$, making $[H^{(s_0)}(X), H^{(s_1)}(X)]$ an admissible pair.

The result stated below is an analogue of Theorem 1 of~\cite{MM-13} concerning the pair $[H^{(s_0)}(\RR^n), H^{(s_1)}(\RR^n)]$, $s_0<s_1$.
For the corresponding result in the case when $X$ is a closed manifold, see Theorem 2.24 of~\cite{MM-14}.

\begin{theorem}\label{T:main-1} Assume that $X$ is a manifold of bounded geometry. Then the following are equivalent:
\begin{enumerate}
  \item [(i)] A Hilbert space $\mathscr{S}$ is an interpolation space with respect to a pair\\
  $[H^{(s_0)}(X), H^{(s_1)}(X)]$, where $s_0<s_1$ are some real numbers.
  \item [(ii)] Up to norm equivalence, we have  $\mathscr{S}=H^{\varphi}(X)$, for some function $\varphi\in RO$ satisfying the condition~(\ref{E:RO-3}).
\end{enumerate}
\end{theorem}

Before stating a corollary of Theorem~\ref{T:main-1}, we recall the definition of an interpolation space with respect to a scale of Hilbert spaces. Let $\{\mathscr{H}_{s}\colon s\in\RR\}$ be a scale of Hilbert spaces such that there is a continuous embedding $\mathscr{H}_{s_1}\hookrightarrow\mathscr{H}_{s_0}$ for all $s_0<s_1$. We say that a Hilbert space $\mathscr{S}$ is an \emph{interpolation space with respect to the scale} $\{\mathscr{H}_{s}\colon s\in\RR\}$ if there exist numbers $s_0<s_1$ such that $\mathscr{S}$ is an interpolation space with respect to the pair $[\mathscr{H}_{s_0},\mathscr{H}_{s_1}]$.

\begin{corollary} \label{C:cor-1} Assume that $X$ is a manifold of bounded geometry. Then the following are equivalent:
\begin{enumerate}
  \item [(i)] A Hilbert space $\mathscr{S}$ is an interpolation space with respect to the scale\\
  $\{H^{(s)}(X)\colon s\in\RR\}$.
  \item [(ii)] Up to the norm equivalence, we have  $\mathscr{S}=H^{\varphi}(X)$, for some function $\varphi\in RO$.
\end{enumerate}
\end{corollary}

\subsection{Second Interpolation Result}\label{SS-1-13} The next result depicts the implication (ii)$\implies$(i) of Theorem~\ref{T:main-1} somewhat more explicitly. The result is an analogue of Theorem 5.1 of~\cite{MM-15}, established for the case $X=\Omega \subset\RR^n$, where $\Omega$ is a bounded domain with Lipschitz boundary. For the corresponding result in the case when $X$ is a closed manifold, see Theorem 2.22 of~\cite{MM-14} (or Theorem 2 of~\cite{MM-09}).

\begin{theorem}\label{T:main-2} Assume that $X$ is a manifold of bounded geometry. Let $\varphi\in RO$, $s_0<\sigma_0(\varphi)$, and $s_1>\sigma_1(\varphi)$. Define $\psi$ as follows:
\begin{equation}\label{E:phi-psi-inverse}
\psi(t):=\left\{\begin{array}{cc}
                  \tau^{-s_0/(s_1-s_0)}\varphi(\tau^{1/(s_1-s_0)}), & \tau\geq 1, \\
                  \varphi(1), & 0<\tau<1.
                \end{array}\right.
\end{equation}
Then, $\psi\in\mathcal{B}$ and $\psi$ is an interpolation parameter. Furthermore, up to norm equivalence, we have
\begin{equation*}
[H^{(s_0)}(X), H^{(s_1)}(X)]_{\psi}=H^{\varphi}(X).
\end{equation*}
\end{theorem}

\subsection{Independence of Geodesic Trivialization}\label{SS-1-13a} As in the case of a compact manifold without boundary (see Theorem 3 in~\cite{MZ-24}), the space $H^{\varphi}(X,\mathscr{T})$ does not depend (up to norm equivalence) on the elements $\gamma_{e_j}$ and $h_j$ used in~(\ref{E:sob-phi-norm-X}):

\begin{theorem}\label{T:ind-triv} Assume that $X$ is a manifold of bounded geometry. Let $\varphi\in RO$. Then the space $H^{\varphi}(X,\mathscr{T})$, defined in Section~\ref{SS-1-10}, is independent (up to norm equivalence) of the choice of geodesic trivialization $\mathscr{T}=(V_j, \gamma_{e_j}, h_j)_{j=1}^{\infty}$ from Section~\ref{SS-1-5}.
\end{theorem}

\subsection{Third Interpolation Result}\label{SS-1-14} The next result shows the closedness of the class $\{H^{\varphi}(X)\colon \varphi\in RO\}$ with respect to interpolation with a function parameter. In the case $X=\Omega \subset\RR^n$, where $\Omega$ is a bounded domain with Lipschitz boundary, an analogous result was established in Theorem 5.2 of~\cite{MM-15}. For the corresponding result in the case when $X$ is a closed manifold, see Theorem 2.22 of~\cite{MM-14} (or Theorem 1 of~\cite{MM-09}).

\begin{theorem}\label{T:main-3} Assume that $X$ is a manifold of bounded geometry. Assume that $\varphi_0,\,\varphi_1\in RO$ and $\frac{\varphi_0}{\varphi_1}$ is bounded in a neighborhood of $\infty$. Let $\psi\in\mathcal{B}$ be an interpolation parameter. Then, the following properties hold:
\begin{enumerate}
  \item [(i)] $[H^{\varphi_0}(X), H^{\varphi_1}(X)]$ is an admissible pair;
  \item [(ii)] up to norm equivalence we have
  \begin{equation}\label{E:t-3-main}
[H^{\varphi_0}(X), H^{\varphi_1}(X)]_{\psi}=H^{\varphi}(X),
\end{equation}
\end{enumerate}
where
\begin{equation}\label{E:quad-int}
\varphi(t):=\varphi_0(t)\psi\left(\frac{\varphi_1(t)}{\varphi_0(t)}\right).
\end{equation}
\end{theorem}

\subsection{Further Properties of $H^{\varphi}(X)$}\label{SS-1-15} In the following theorem we list further properties of the class $\{H^{\varphi}(X)\colon \varphi\in RO\}$. Analogous properties in the case $X=\RR^n$ were established in Proposition 2 of~\cite{MM-13} (or Proposition 2.6 in~\cite{MM-14}).

\begin{theorem}\label{T:main-4} Assume that $X$ is a manifold of bounded geometry. Then, we have the following properties:
\begin{enumerate}

\item [(i)] Assume that $\varphi\in RO$. Then $\xcomp$ is dense in $H^{\varphi}(X)$.

\item [(ii)] Assume that $\varphi_0,\,\varphi_1\in RO$ and $\frac{\varphi_0}{\varphi_1}$ is bounded in a neighborhood of $\infty$. Then we have a continuous embedding $H^{\varphi_1}(X)\hookrightarrow H^{\varphi_0}(X)$.

\item [(iii)] Assume that $\varphi\in RO$. Then, the sesquilinear form~(\ref{E:inner-l-2-x}) extends to a sesquilinear duality (separately continuous sesquilinear form)
\begin{equation}\label{E:inner-product-h-phi}
(\cdot,\cdot)\colon H^{\varphi}(X)\times H^{\frac{1}{\varphi}}(X)\to\mathbb{C}.
\end{equation}
The spaces $H^{\varphi}(X)$ and $H^{\frac{1}{\varphi}}(X)$ are mutually dual relative to the duality~(\ref{E:inner-product-h-phi}).
\end{enumerate}
\end{theorem}

Before stating our next result, we review the notion of proper uniform pseudo-differential operator on $X$.

\subsection{H\"ormander Class $S^{m}(\RR^{n}\times \RR^{n})$}\label{SS-1-16} We begin by recalling a special case of the (uniform) H\"ormander symbol class:
\begin{definition}\label{D-hormander-class} Let $m\in\RR$. The notation $S^{m}(\RR^{n}\times \RR^{n})$ refers to the set of (complex-valued) functions $a(x,\xi)\in C^{\infty}(\RR^{2n})$ such that the following property is satisfied: for all $\alpha,\,\beta\in \NN_0^{n}$, there exists a constant $C_{\alpha,\beta}>0$ such that
\begin{equation}\label{E:est-m-2}
 |\p_{\xi}^{\alpha}\p_{x}^{\beta}a(x,\xi)|\leq C_{\alpha,\beta}\langle\xi\rangle^{m-|\alpha|},
\end{equation}
for all $x,\,\xi\in \RR^n$.
\end{definition}

\begin{remark}\label{R-1} To simplify our presentation, in this article we use $S^{m}(\RR^{n}\times \RR^{n})$, which is a special case $S^{m}_{1,0}(\RR^{n}\times \RR^{n})$ of the more general (uniform) H\"ormander class $S^{m}_{\rho,\delta}(\RR^{n}\times \RR^{n})$, $0\leq\rho,\delta\leq 1$, in which the right hand side of the above inequality is replaced by a more general term $\langle\xi\rangle^{m-\rho|\alpha|+\delta|\beta|}$. It turns out that Theorem~\ref{T:main-2} below remains valid for operators corresponding to the symbol class $S^{m}_{\rho,\delta}(\RR^{n}\times \RR^{n})$ with $0<\delta<\rho<1$.
\end{remark}

\subsection{Fourier Transform and Pseudo-Differential Operator with $S^{m}(\RR^{n}\times \RR^{n})$-type Symbol}\label{SS-1-17}
Recall the Schwartz space $S(\RR^n)$ of rapidly decreasing functions and the Fourier transform $\widehat{u}$ of a function $u\in S(\RR^n)$:
For a symbol $a\in S^{m}(\RR^{n}\times \RR^{n})$ we define the pseudo-differential operator $\textrm{Op}[a]$ (denoted interchangeably in this paper as $T_{a}$) with the help of the formula:
\begin{equation}\label{E:op-a}
  (T_{a}u)(x):= (2\pi)^{-n}\int_{{\RR}^{n}} e^{ix\cdot\xi}a(x,\xi)\widehat{u}(\xi)\,d\xi,\quad u\in S(\RR^n).
\end{equation}
As indicated in~\cite{W-pdo-book}, the mapping $T_{a}\colon S(\RR^n)\to  S(\RR^n)$ is a continuous linear operator, which extends to a continuous linear operator $T_{a}\colon S'(\RR^n)\to  S'(\RR^n)$, where $S'(\RR^n)$ is the space of tempered distributions.

\subsection{Proper Uniform PDO Classes}\label{SS-1-18} It turns out that proper uniform pseudo-differential operators (PUPDO) on manifolds of bounded geometry posses the usual properties encountered in the calculus based on H\"ormander-type symbols on $\RR^n$. The theory of PUPDO was developed by the authors of~\cite{MS-1} in the setting of unimodular Lie groups. Subsequently, this theory was extended to manifolds of bounded geometry in~\cite{Kor-91}. Here we summarize the definitions of various operator classes in the context of a manifold of bounded geometry $X$. In our presentation, we follow Section 1 of~\cite{MS-1} and Section 2 of~\cite{Kor-91}.

\begin{definition}\label{D-1-x} The class $U\Psi^{-\infty}(X)$ consists of operators $R$ whose kernels $K_{R}\in C^{\infty}_{b}(X\times X)$ satisfy the following property: there exists a number $c_{R}>0$ such that $K_{R}(x,y)=0$ for all $x,\,y$ with $d_{g}(x,y)>c_{R}$.
\end{definition}

In the discussion below, the notations $B$, $e\in O(X)$, and $\gamma_{e}\colon B\to X$ are as in Remark~\ref{R:b-geo}.


\begin{definition}\label{D-2-x} Let $m\in\RR$. The class $FS^{m}(B)$ is comprised of the families $\{a_{e}\in C^{\infty}(B\times \RR^n)\colon e\in O(X)\}$ obeying the estimates~(\ref{E:est-m-2}), where the constants $C_{\alpha,\beta}$ do not depend on the frames $e$, and the wording ``for all $x\in\RR^n$" is replaced by ``for all $x\in B$."
\end{definition}

Given an element $\{a_{e}\in C^{\infty}(B\times \RR^n)\colon e\in O(X)\}\in FS^{m}(B)$, the formula~(\ref{E:op-a}) gives rise to a family of operators $T_{a_{e}}\colon C_{c}^{\infty}(B) \to C^{\infty}(B)$ indexed by $e\in O(X)$.

\begin{definition}\label{D-3-x} The class $\fpi(B)$ consists of the families $\{R_{e}\colon e\in O(X)\}$ of operators $R_{e}\colon C_{c}^{\infty}(B) \to C^{\infty}(B)$ whose kernels  $K_{R_{e}}$ belong to $C^{\infty}(B\times B)$ and obey the following property: for all $\alpha,\,\beta\in \NN_0^{n}$, there exists a constant $C_{\alpha,\beta}>0$ such that
\begin{equation}\label{E:est-m-3}
  |\p^{\alpha}_{x}\p^{\beta}_{y}K_{e}(x,y)|\leq C_{\alpha,\beta},
\end{equation}
for all $x,\,y\in B$, where the constants $C_{\alpha,\beta}>0$ do not depend on the frames $e$.
\end{definition}

The following definition combines the objects from Definition~\ref{D-2-x} and Definition~\ref{D-3-x}.

\begin{definition}\label{D-4-x} Let $m\in\RR$. The elements of the class $\fprl^{m}(B)$ are the families $\{A_{e}\colon e\in O(X)\}$ of operators $A_{e}\colon C_{c}^{\infty}(B) \to C^{\infty}(B)$ satisfying the following property: $A_{e}=T_{a_{e}}+R_{e}$, with $\{a_{e}\in C^{\infty}(B\times \RR^n)\colon e\in O(X)\}\in FS^{m}(B)$ and $\{R_{e}\colon e\in O(X)\}\in\fpi(B)$.
\end{definition}

In the definition below, the notation $\Delta_{X}$ indicates the set $\{(x,x)\in X\times X\colon x\in X\}$.

\begin{definition}\label{D-5-x} Let $m\in\RR$. The notation $\uprl^{m}(X)$ refers to the set of operators $A\colon C_{c}^{\infty}(X) \to C_{c}^{\infty}(X)$ satisfying the following conditions:
\begin{enumerate}
  \item [(i)] There exists a constant $c_{A}>0$ (depending on $A$) such that $K_{A}(x,y)=0$ for all $x,\,y$ with $d_{g}(x,y)>c_{A}$, where $K_{A}$ is the operator kernel of $A$.
  \item [(ii)] The kernel $K_{A}$ belongs to $C^{\infty}(X\times X\backslash\Delta_X)$ and satisfies the following condition:  for all $k_1,k_2\in\NN$ and all $k_3>0$, there exists a constant $C_{k_1,k_2,k_3}>0$ such that
      \begin{equation*}
        |\nabla_{x}^{k_1}\nabla_{y}^{k_2}K_{A}(x,y)|\leq C_{k_1,k_2,k_3},
      \end{equation*}
      for all $x,\,y$ with $d_{g}(x,y)>k_3$, where $\nabla^{k_j}$ is the $k_j$-th covariant derivative corresponding to the Levi--Civita connection.
  \item [(iii)] Let $e\in O(X)$ be a frame such that $\pi(e)=x$ and let $B$, $\tau_0$ be as in Remark~\ref{R:b-geo}. Let $A_{e}$ be the operator defined by the commutative diagram
\[
\begin{array}{ccc}
C_{c}^{\infty}(B(x,\tau_0)) & \xrightarrow{A} & C^{\infty}(B(x,\tau_0)) \\
\downarrow{\gamma_{e}^*} &  & \downarrow{\gamma_{e}^*} \\
C_{c}^{\infty}(B) & \xrightarrow{A_{e}} & C^{\infty}(B),
\end{array}
\]
where $(\gamma_{e}^*v)(y):=(v\circ \gamma_{e})(y)$. Then, the operator family $\{A_{e}\colon e\in O(X)\}$ belongs to $\fprl^{m}(B)$.
\end{enumerate}
\end{definition}
\begin{definition}\label{D-6-x-bdd} Let $m\in\RR$. A family $\mathscr{F}$ of operators in $\uprl^{m}(X)$ is said to be bounded in $\uprl^{m}(X)$ if the following two properties are satisfied:
\begin{enumerate}
  \item [(i)] For every $A\in\mathscr{F}$ conditions (i) and (ii) of Definition~\ref{D-5-x} are satisfied, with the constants independent of $A\in \mathscr{F}$.
  \item [(ii)] Let $e\in O(X)$ be a frame such that $\pi(e)=x$ and let $B$, $\tau_0$ be as in Remark~\ref{R:b-geo}. Let $A_{e}$ be the operator defined by the commutative diagram from part (iii) of Definition~\ref{D-5-x}. Then, the set $\widetilde{\mathscr{F}}:=\{\{A_{e}\colon e\in O(X)\}\colon A\in \mathscr{F}\}$  is bounded in $\fprl^{m}(B)$, that is, the constants in the estimates~(\ref{E:est-m-2}) and~(\ref{E:est-m-3}) are independent of the members of the set $\widetilde{\mathscr{F}}$.
\end{enumerate}
\end{definition}
Lastly, we recall the definition of elliptic PUPDO class:

\begin{definition}\label{D-6-x} Let $m\in\RR$. The class $\ueprl^{m}(X)$ consists of the operators $A$ with the following properties:
\begin{enumerate}
  \item [(i)] $A\in \uprl^{m}(X)$.
  \item [(ii)] Let $A_{e}$ be as in part (iii) of Definition~\ref{D-5-x} and let $\{a_{e}\colon e\in O(X)\}$ correspond to $A_{e}$ as in Definition~\ref{D-4-x}. Then there exist constants $C$ and $R$ such that
      \begin{equation}\label{E:elliptic-def-loc}
        |a_{e}(x,\xi)|\geq C|\xi|^{m},
       \end{equation}
for all $|\xi|>R$, all $x\in B$, and all $e\in O(X)$. (Here, $C$ and $R$ are independent of $x$, $\xi$, and $e\in O(X)$.)
\end{enumerate}
\end{definition}

\subsection{Mapping Property for PUPDO}\label{SS-1-19} We are ready to formulate a mapping property of an operator $A\in \uprl^{m}(X)$ with respect to the scale $\{H^{\varphi}(X)\colon\varphi\in RO\}$. The result stated below  is an analogue of Theorem 3 in~\cite{MM-09} concerning the case when $X$ is a closed manifold (see also Section 4 of~\cite{MM-13} for the corresponding result in the case $X=\RR^n$).

\begin{theorem}\label{T:main-5} Assume that $X$ is a manifold of bounded geometry. Assume that $A\in \uprl^{m}(X)$, where $m\in\RR$. Let $\varphi\in RO$. Then $A$ extends to a bounded linear operator
\begin{equation}\nonumber
A\colon H^{\varphi}(X)\to H^{t^{-m}\varphi}(X),
\end{equation}
where $t^{-m}\varphi$ indicates the product of the functions $t^{-m}$ and $\varphi$.

Furthermore, if $\mathscr{F}$ is a bounded family in $\uprl^{m}(X)$, then there exists a constant $C>0$ (depending on $\varphi\in RO$ but not on $A\in\mathscr{F}$) such that
\begin{equation}\label{E:bdd-A-fm}
\|A\|_{H^{\varphi}(X)\to H^{t^{-m}\varphi}(X)}\leq C,
\end{equation}
for all $A\in\mathscr{F}$.
\end{theorem}
\subsection{An Embedding Result}\label{SS-1-20} The following embedding theorem is an analogue of part (vi) of Proposition 2.6 in~\cite{MM-14} for the case $X=\RR^n$. (See also Proposition 7.6 in~\cite{MM-21}.)
\begin{theorem}\label{T:main-6} Assume that $X$ is a manifold of bounded geometry. Let $k\in\NN_{0}$. Assume that $\varphi\in RO$ and
\begin{equation}\label{E:emb-cb-k}
\int_{1}^{\infty}\frac{t^{2k+n-1}}{\varphi^2(t)}\,dt<\infty.
\end{equation}
Then, we have a continuous embedding $H^{\varphi}(X)\hookrightarrow C_{b}^{k}(X)$, where $C_{b}^{k}(X)$ is as in Section~\ref{SS-1-6}.
\end{theorem}

Before stating the last result we recall the definition of the so-called (extended) Hilbert scale.
\subsection{Extended Hilbert scale}\label{SS-1-21} We begin by writing down some terminology from Section 2 of~\cite{MM-21}.
Let $\mathscr{H}$ be a separable complex Hilbert space with inner product $(\cdot,\cdot)_{\mathscr{H}}$ and norm $\|\cdot\|_{\mathscr{H}}$. Let $A$ be a self-adjoint operator in $\mathscr{H}$ satisfying $(Au,u)_{\mathscr{H}}\geq\|u\|^2_{\mathscr{H}}$ for all $u\in\dom (A)$.

With this description of $A$, using spectral calculus we define the operator $A^{s}$ for each $s\in\RR$. Note that $\dom(A^s)$ is dense in $\mathscr{H}$; in particular, if $s\leq0$ we have  $\dom (A^s)=\mathscr{H}$. We define $H_{A}^{(s)}$ as the completion of $\dom(A^s)$ with respect to the inner product
\begin{equation*}
(u,v)_{s}:=(A^su,A^sv)_{\mathscr{H}}, \qquad u,v\in \dom(A^s).
\end{equation*}
It turns out that $H_{A}^{(s)}$ is a separable Hilbert space whose inner product and norm will be denoted by $(\cdot,\cdot)_{s}$ and $\|\cdot\|_{s}$. We call $\{H_{A}^{(s)}\colon s\in\RR\}$ \emph{Hilbert scale generated by $A$} or, in short form, \emph{$A$-scale}. As indicated in Section 2 of ~\cite{MM-21}, for $s\geq 0$ we have $H_{A}^{(s)}=\dom (A^s)$, while for $s<0$ we have $H_{A}^{(s)}\supset \mathscr{H}$.

As explained in Section 2 of~\cite{MM-21}, for all $s_0<s_1$,  $[H_{A}^{(s_0)},H_{A}^{(s_1)}]$ constitutes an admissible pair (in the sense of Section~\ref{SS-2-1} above). Recalling the definition of interpolation space (see Section~\ref{SS-1-11} above), the term \emph{extended Hilbert scale generated by $A$} or
\emph{extended $A$-scale} refers to the set of all Hilbert spaces that serve as interpolation spaces with respect pairs of the form $[H_{A}^{(s_0)},H_{A}^{(s_1)}]$, $s_0<s_1$.

Recalling the hypotheses on  $A$ and using spectral calculus, for a Borel measurable function $\varphi\colon [1,\infty)\to(0,\infty)$,  we can define a (positive self-adjoint) operator $\varphi(A)$ in $\mathscr{H}$. Furthermore, we define $H_{A}^{\varphi}$ as the completion of
$\dom(\varphi(A))$ with respect to the inner product
\begin{equation*}
(u,v)_{\varphi}:=(\varphi(A)u,\varphi(A)v)_{\mathscr{H}}, \qquad u,v\in \dom(\varphi(A)).
\end{equation*}
It turns out (see Section 2 of~\cite{MM-21}) that $H_{A}^{\varphi}$ is a separable Hilbert space, and its inner product and norm will be denoted by $(\cdot,\cdot)_{\varphi}$ and $\|\cdot\|_{\varphi}$. Furthermore, as indicated in Section 2 of~\cite{MM-21}, we have $H_{A}^{\varphi}=\dom(\varphi(A))$ if and only if $0\notin\textrm{Spec}(\varphi(A))$.

\subsection{Generalized $A$-scale Property}\label{SS-1-22}
We begin with a list of assumptions on our operator:
\begin{enumerate}
\item [(H1)] $A\in\ueprl^{1}(X)$, where $\ueprl^{1}(X)$ is as in Definition~\ref{D-6-x};

\item [(H2)] $(Au,u)\geq \|u\|^2$, for all $u\in\xcomp$, where $(\cdot,\cdot)$ is as in~(\ref{E:inner-l-2-x}) and $\|\cdot\|$ is the corresponding norm in $L^2(X)$.
\end{enumerate}

\begin{remark}\label{R:kor} Under the assumptions (H1)--(H2), $A|_{\xcomp}$ is a formally self-adjoint operator belonging to the class $\ueprl^{1}(X)$. Therefore, by Corollary 3.15 in~\cite{Kor-91}, the operator $A|_{\xcomp}$ is essentially self-adjoint in $L^2(X)$.  Furthermore, by the mentioned result from~\cite{Kor-91}, the domain of the self-adjoint closure of $A$ is the space $H^{(1)}(X)$. For notational simplicity, we denote the self-adjoint closure of $A|_{\xcomp}$ again by $A$. With this clarification, as in Section~\ref{SS-1-21} above (now in the setting $\mathscr{H}=L^2(X)$), for a Borel function
$\varphi\colon [1,\infty)\to(0,\infty)$ we can define the space $H_{A}^{\varphi}$ corresponding to an operator $A$ satisfying (H1)--(H2).
\end{remark}

The following theorem is an analogue of Theorem 5.3 in~\cite{MM-21} for the case when $X$ is a closed manifold.
\begin{theorem}\label{T:main-7} Assume that $X$ is a manifold of bounded geometry. Let $\varphi\in RO$. Furthermore, assume that $A$ is an operator satisfying (H1)--(H2). Let $H^{\varphi}_{A}(X)$ is as in Remark~\ref{R:kor} and $H^{\varphi}(X)$ is as in Section~\ref{SS-1-10}.

Then, up to norm equivalence, we have
\begin{equation}\label{E:thm-7}
H^{\varphi}_{A}(X)=H^{\varphi}(X).
\end{equation}

\end{theorem}
\begin{remark} The condition $A\in\ueprl^{1}(X)$ in (H1) can be replaced by $A\in\ueprl^{m}(X)$, with $m>0$. It turns out that in this situation the following variant of~(\ref{E:thm-7}) holds: $H^{\varphi}_{A}(X)=H^{\varphi_{m}}(X)$, where $\varphi_{m}(t):=\varphi(t^m)$, $t\geq 1$.
\end{remark}

\section{Preliminary Facts}\label{S:S-2} In this section we collect some interpolation results from~\cite{MM-08,MM-13, MM-14} used in the proof of Theorem~\ref{T:main-1}.

For future reference, we say that a Borel measurable function $\psi\colon (0,\infty)\to (0,\infty)$ is \emph{pseudoconcave in a neighborhood of} $\infty$ if there exists a concave function $\psi_1 \colon (c,\infty)\to (0,\infty)$, where $c>1$ is a large number, such that $\frac{\psi}{\psi_1}$ and $\frac{\psi_1}{\psi}$ are bounded on $(c,\infty)$.

We begin with a proposition for which we refer to Theorem 1.9 in~\cite{MM-14}.
\begin{prop}\label{L-1} A function $\psi\in\mathcal{B}$ is an interpolation parameter if and only if $\psi$ is pseudoconcave in a neighborhood of $\infty$.
\end{prop}

The next proposition (see Theorem 11.4.1 in~\cite{O-84}) is instrumental in proving the implication (i)$\implies$(ii) of Theorem~\ref{T:main-1}.

\begin{prop}\label{L-2} Let $\mathscr{H}=[\mathscr{H}_{0},\mathscr{H}_{1}]$ be an admissible pair of Hilbert spaces. Assume that a space $\mathscr{S}$ is an interpolation space for the pair $\mathscr{H}$. Then, up to norm equivalence, we have $\mathscr{S}=\mathscr{H}_{\psi}$, for some function $\psi\in\mathcal{B}$ such that $\psi$ is pseudoconcave in a neighborhood of $\infty$.
\end{prop}

The next result (see Theorem 4.1 in~\cite{MM-15}) enables us to reduce Theorem~\ref{T:main-1} to the $\RR^n$ situation, after applying a suitable localization procedure.

\begin{prop}\label{L-3} Let $s_0<s_1$ be two real numbers and let $\psi\in\mathcal{B}$ be an interpolation parameter. Define
\begin{equation}\label{E:phi-psi}
\varphi(t):=t^{s_0}\psi(t^{s_1-s_0}),\quad t\geq 1.
\end{equation}
Then the following properties hold:
\begin{enumerate}
  \item [(i)] $\varphi\in RO$;
  \item [(ii)] we have (with equality of norms)
\begin{equation}\label{interp-r-n}
[H^{(s_0)}(\RR^n), H^{(s_1)}(\RR^n)]_{\psi}=H^{\varphi}(\RR^n).
\end{equation}
\end{enumerate}
\end{prop}

For the proof of the following proposition, see Theorem 4.2 in~\cite{MM-15}:

\begin{prop}\label{L-3-b} Let $s_0<s_1$ be two real numbers and let $\psi\in\mathcal{B}$. Assume that $\varphi$ is defined by the formula~(\ref{E:phi-psi}). Then the following are equivalent:
\begin{enumerate}
  \item [(i)] $\psi$ is an interpolation parameter;
  \item [(ii)] $\varphi$ satisfies~(\ref{E:RO-3}) with some constant $c\geq 1$ independent of $t$ and $\lambda$.
\end{enumerate}
\end{prop}

The key instrument for justifying the implication (ii)$\implies$(i) of Theorem~\ref{T:main-1} will be the next proposition (see the proof of the sufficiency part of Theorem 2.4 in~\cite{MM-15}):

\begin{prop}\label{L-4} Assume that $\varphi\in RO$ satisfies the condition~(\ref{E:RO-3}). Define $\psi$ as in~(\ref{E:phi-psi-inverse}),
where $s_0<s_1$ are as in~(\ref{E:RO-3}). Then the following properties hold:
\begin{enumerate}
  \item [(i)] $\psi\in \mathcal{B}$ and $\psi$ satisfies~(\ref{E:phi-psi});
  \item [(ii)] $\psi$ is an interpolation parameter.
\end{enumerate}
\end{prop}

\section{Proof of Theorem~\ref{T:main-1}}\label{S:S-3}

We start with three auxiliary lemmas.

\subsection{Auxiliary Lemmas}

For the following result, see Lemma 2.1 from~\cite{G-13}:

\begin{lemma}\label{L:Triebel} In both parts of the lemma, we assume that $s\in\RR$.
\\\\
\noindent (i) Let $u\in H^{s}(\RR^n)$ and let $\chi\in C^{\infty}(\RR^n)$ be a function such that for all $\alpha\in\NN_{0}^{n}$ with $|\alpha|\leq |s|+1$ we have
\begin{equation*}
|\p^{\alpha}\chi(x)|\leq C_{|\alpha|},
\end{equation*}
for all $x\in X$, where $C_{|\alpha|}$ are constants.

Then, there exists a constant $C$, depending only on $s$ and $n$ and $C_{|\alpha|}$, such that
\begin{equation}\label{E:triebel-1}
\|\chi u\|_{H^{s}(\RR^n)}\leq C\|u\|_{H^{s}(\RR^n)}.
\end{equation}
\noindent (ii) Let $u\in H^{s}(\RR^n)$  with $\supp u \subset Y_1\subset\RR^n$, where $Y_1$ is an open set. Let $\tau\colon Y_2 \subset \RR^n  \to Y_1 \subset \RR^n$  be a diffeomorphism such that for all $\alpha\in\NN_{0}^{n}$ with $|\alpha|\leq |s|+1$ we have
\begin{equation*}
|\p^{\alpha}\tau(x)|\leq C_{|\alpha|},
\end{equation*}
for all $x\in X$, where $C_{|\alpha|}$ are constants.

Then, there exists a constant $C$, depending only on $s$ and $n$ and $C_{|\alpha|}$, such that
\begin{equation}\label{E:triebel-2}
\| u\circ\tau\|_{H^{s}(\RR^n)}\leq C\|u\|_{H^{s}(\RR^n)}.
\end{equation}
\end{lemma}

Before stating the next lemma, we describe one more notation. For a Hilbert space $\mathscr{H}$ with norm $\|\cdot\|_{\mathscr{H}}$, define $\ell_2(\mathscr{H})$ as the  space of all sequences $u:=(u_k)_{k=1}^{\infty}$, with $u_k\in \mathscr{H}$, such that
\begin{equation}\label{E:ell-2-1}
\|u\|_{\ell_2(\mathscr{H})}^2:=\sum_{k=1}^{\infty}\|u_k\|^2_{\mathscr{H}}<\infty.
\end{equation}

We record the following special case of Theorem 1.5 in~\cite{MM-14} (or Theorem 2.5 of~\cite{MM-08}):

\begin{lemma}\label{L-6} Let $[\mathscr{H}_{0},\mathscr{H}_{1}]$ be an admissible pair of Hilbert spaces and let $\psi\in\mathcal{B}$. Then, the following properties hold:
\begin{enumerate}
  \item [(i)] $[\ell_2(\mathscr{H}_{0}),\ell_2(\mathscr{H}_{1})]$ is also an admissible pair;
  \item [(ii)] we have
\begin{equation}\label{interp-r-n-l-2}
[\ell_2(\mathscr{H}_{0}),\ell_2(\mathscr{H}_{1})]_{\psi}=\ell_2([\mathscr{H}_{0},\mathscr{H}_{1}]_{\psi}),
\end{equation}
with equality of norms.
\end{enumerate}
\end{lemma}

We now come to an important ingredient of the proof of Theorem~\ref{T:main-1}:

\begin{lemma}\label{L:G-bounded} Let $X$ be a manifold of bounded geometry, and let $\mathscr{T}=(V_j, \gamma_{e_j}, h_j)_{j=1}^{\infty}$ be a geodesic trivialization as in Section~\ref{SS-1-5}. Let $\mathcal{A}(j):=\{k\colon V_{j}\cap V_{k}\neq \emptyset\}$ and let
\begin{equation*}
H_j:=\left(\sum_{k\in \mathcal{A}(j)}h_k\right)\circ \gamma_{e_j}.
\end{equation*}

Let $\varphi\in RO$. Define a (linear) operator $G$ acting on $(v_j)_{j=1}^{\infty}\in \ell^2(H^{\varphi}(\RR^n))$ as follows:
\begin{equation}\label{E:op-g}
G ((v_j)_{j=1}^{\infty}):=\sum_{j=1}^{\infty}(H_jv_j)\circ \kappa_{e_j},
\end{equation}
where $\kappa_{e_j}$ is as in Section~\ref{SS-1-5} and $(H_jv_j)\circ \kappa_{e_j}$ is extended by zero.

Then,
\begin{equation}\label{E:op-g-bdd}
G\colon\ell^2(H^{\varphi}(\RR^n)) \to H^{\varphi}(X)
\end{equation}
is a bounded linear operator.
\end{lemma}
\begin{proof} With $\varphi\in RO$, we choose numbers $s_0,\,s_1$ such that $s_0<\sigma_0(\varphi)$ and $s_1>\sigma_1(\varphi)$.

We first estimate
\begin{align}\label{E:est-1-Triebel-1}
&\|\left(h_i(G ((v_j)_{j=1}^{\infty}))\right)\circ \gamma_{e_i}\|^2_{H^{s_0}(\RR^n)}\nonumber\\
&=\left\|\left(h_i\left(\sum_{j=1}^{\infty}(H_jv_j)\circ
\kappa_{e_j}\right)\right)\circ\gamma_{e_i}\right\|^2_{H^{s_0}(\RR^n)}\nonumber\\
&=\left\|\left(h_i\left(\sum_{j\in \mathcal{A}(i)}(H_jv_j)\circ \kappa_{e_j}\right)\right)\circ\gamma_{e_i}\right\|^2_{H^{s_0}(\RR^n)}\nonumber\\
&\leq \widetilde{C_1}\sum_{j\in \mathcal{A}(i)} \left\|[h_i\circ\gamma_{e_i}]\left[(H_jv_j)\circ(\kappa_{e_j}\circ\gamma_{e_i})\right]\right\|^2_{H^{s_0}(\RR^n)},
\end{align}
where $\widetilde{C_1}$ is a constant.

The second equality in~(\ref{E:est-1-Triebel-1}) holds since $\supp h_{i}\subset V_{i}$ (see property (C2)(i) in Section~\ref{SS-1-4}). Furthermore, due to the uniform local finiteness of the cover $\{V_{i}=B(x_i,\varepsilon)\}$, as discussed in property (C1) of Section~\ref{SS-1-4}, the constant $\widetilde{C_1}$ in~(\ref{E:est-1-Triebel-1}) can be chosen independently of the index $i$.

Using part (i) of Lemma~\ref{L:Triebel}, we estimate
\begin{align}\label{E:est-1-Triebel-2}
&\widetilde{C_1}\sum_{j\in \mathcal{A}(i)}
\left\|[h_i\circ\gamma_{e_i}]\left[(H_jv_j)\circ(\kappa_{e_j}\circ\gamma_{e_i})\right]\right\|^2_{H^{s_0}(\RR^n)}\nonumber\\
&\leq \widetilde{C_2}\sum_{j\in \mathcal{A}(i)} \left\|(H_jv_j)\circ(\kappa_{e_j}\circ\gamma_{e_i})\right\|^2_{H^{s_0}(\RR^n)},
\end{align}
where $\widetilde{C_2}$ is a constant.

Thanks to the property~(\ref{E:geodesic-atlas-h}) and Remark~\ref{R:geo-atlas-h}, the constant $\widetilde{C_2}$ can be chosen independently of the index $i$.

Next, using part (ii) of Lemma~\ref{L:Triebel}, we estimate
\begin{equation}\label{E:est-1-Triebel-3}
\widetilde{C_2}\sum_{j\in \mathcal{A}(i)} \left\|(H_jv_j)\circ(\kappa_{e_j}\circ\gamma_{e_i})\right\|^2_{H^{s_0}(\RR^n)}\nonumber\\
\leq \widetilde{C_3}\sum_{j\in \mathcal{A}(i)} \left\|H_jv_j\right\|^2_{H^{s_0}(\RR^n)},
\end{equation}
where $\widetilde{C_3}$ is a constant.

Taking into account the property~(\ref{E:geo-atlas-i-j}) and Remark~\ref{R:geo-atlas-i-j}, the constant $\widetilde{C_3}$ can be chosen independently of the index $i$.

Furthermore, using part (i) of Lemma~\ref{L:Triebel}, we estimate
\begin{equation}\label{E:est-1-Triebel-4}
\widetilde{C_3}\sum_{j\in \mathcal{A}(i)} \left\|H_jv_j\right\|^2_{H^{s_0}(\RR^n)}\leq \widetilde{C_4}\sum_{j\in \mathcal{A}(i)}\|v_j\|^2_{H^{s_0}(\RR^n)},
\end{equation}
where $\widetilde{C_4}$ is a constant.

Note that the properties  (C1) and (C2) of Section~\ref{SS-1-4} ensure that the functions $H_j$ satisfy the hypotheses of part (i) of Lemma~\ref{L:Triebel} and that the constant $\widetilde{C_4}$ can be chosen independently of the index $i$.

Putting together~(\ref{E:est-1-Triebel-1})--(\ref{E:est-1-Triebel-4}), we get
\begin{equation}\label{E:est-1-Triebel-5}
\|\left(h_i(G ((v_j)_{j=1}^{\infty}))\right)\circ \gamma_{e_i}\|^2_{H^{s_0}(\RR^n)}
\leq \widetilde{C_4}\sum_{j\in \mathcal{A}(i)}\|v_j\|^2_{H^{s_0}(\RR^n)},
\end{equation}
where $\widetilde{C_4}$ is a constant independent of the index $i$.

In the same way as in~(\ref{E:est-1-Triebel-5}), for $s_1>\sigma_{1}(\varphi)$ we obtain
\begin{equation}\label{E:est-1-Triebel-6}
\|\left(h_i(G ((v_j)_{j=1}^{\infty}))\right)\circ \gamma_{e_i}\|^2_{H^{s_1}(\RR^n)}
\leq \widetilde{C_5}\sum_{j\in \mathcal{A}(i)}\|v_j\|^2_{H^{s_1}(\RR^n)},
\end{equation}
where $\widetilde{C_5}$ is a constant independent of the index $i$.

Since $\supp h_{i}\subset V_{i}$, the action of the operator $Z_{i}(\cdot):=(h_iG(\cdot))\circ \gamma_{e_i}$ on a sequence $(v_j)_{j=1}^{\infty}\in \ell^2(H^{\varphi}(\RR^n))$ is the same as the action of $Z_{i}$ on the sequence $(w_j)_{j=1}^{\infty}$, where $w_j=v_j$ if $j\in \mathcal{A}(i)$ and $w_j=0$ if $j\notin \mathcal{A}(i)$.

We denote by $\ell^2_{[i]}(H^{\varphi}(\RR^n))$, $\varphi\in RO$, the space of sequences $(w_j)_{j=1}^{\infty}\in \ell^2(H^{\varphi}(\RR^n))$ such that $w_j=0$ if $j\notin \mathcal{A}(i)$. We equip this space with the norm
\[
\|(w_j)_{j=1}^{\infty}\|^2_{\ell^2_{[i]}(H^{\varphi}(\RR^n))}:=\sum_{j\in \mathcal{A}(i)}\|w_j\|^2_{H^{\varphi}(\RR^n)}.
\]

Keeping in mind the notation $\ell^2_{[i]}(H^{\varphi}(\RR^n))$ (specialized to the case $\varphi(t)=t^{s_0}$ and $\varphi(t)=t^{s_1}$), the estimates~(\ref{E:est-1-Triebel-5}) and~(\ref{E:est-1-Triebel-6}) show that
\[
Z_i\colon \ell^2_{[i]}(H^{s_0}(\RR^n))\to H^{s_0}(\RR^n),\qquad Z_i\colon \ell^2_{[i]}(H^{s_1}(\RR^n))\to H^{s_1}(\RR^n)
\]
are bounded linear operators.

Since $s_0<\sigma_0(\varphi)$ and $s_1>\sigma_1(\varphi)$, the function $\varphi$ satisfies the condition~(\ref{E:RO-3}). Defining $\psi$ as in~(\ref{E:phi-psi-inverse}) and using Proposition~\ref{L-4}, we can see that $\psi\in\mathcal{B}$, $\psi$ is an interpolation parameter, and $\psi$ satisfies~(\ref{E:phi-psi}). Thus, by Proposition~\ref{L-3}, we get a bounded linear operator
\begin{equation}\label{E:est-1-Triebel-7a}
Z_i\colon [\ell^2_{[i]}(H^{s_0}(\RR^n)), \ell^2_{[i]}(H^{s_1}(\RR^n))]_{\psi}\to H^{\varphi}(\RR^n).
\end{equation}

Note that the space $\ell^2_{[i]}(H^{\varphi}(\RR^n))$ can be viewed as
\[
\displaystyle\bigoplus_{j\in \mathcal{A}(i)}H^{\varphi}(\RR^n),
\]
where the sum has finitely many terms (because $\mathcal{A}(i)$ is a finite set).

Using Lemma~\ref{L-6} and Proposition~\ref{L-3}, we have
\begin{align}\label{E:est-1-Triebel-7b}
&[\ell^2_{[i]}(H^{s_0}(\RR^n)), \ell^2_{[i]}(H^{s_1}(\RR^n))]_{\psi}=\bigoplus_{j\in \mathcal{A}(i)}[H^{s_0}(\RR^n),H^{s_1}(\RR^n)]_{\psi}\nonumber\\
&=\bigoplus_{j\in \mathcal{A}(i)}H^{\varphi}(\RR^n)=\ell^2_{[i]}(H^{\varphi}(\RR^n)).
\end{align}
Taking into account~(\ref{E:est-1-Triebel-7a}) and~(\ref{E:est-1-Triebel-7b}) and using Proposition~\ref{L-3}, we get a bounded linear operator
\[
Z_i\colon \ell^2_{[i]}(H^{\varphi}(\RR^n))\to H^{\varphi}(\RR^n).
\]
Therefore, the estimates~(\ref{E:est-1-Triebel-5}) and~(\ref{E:est-1-Triebel-6}) lead to
\begin{equation}\label{E:est-1-Triebel-7}
\|\left(h_i(G ((v_j)_{j=1}^{\infty}))\right)\circ \gamma_{e_i}\|^2_{H^{\varphi}(\RR^n)}
\leq \widetilde{C_6}\sum_{j\in \mathcal{A}(i)}\|v_j\|^2_{H^{\varphi}(\RR^n)},
\end{equation}
where $\widetilde{C_6}$ is a constant.

As the constants $\widetilde{C_4}$ and $\widetilde{C_5}$ in ~(\ref{E:est-1-Triebel-5}) and~(\ref{E:est-1-Triebel-6}) do not depend on the index $i$, Remark~\ref{R-rem-int} ensures that constant $\widetilde{C_6}$ is independent of $i$.

Keeping in mind the estimate~(\ref{E:est-1-Triebel-7}) and referring to~(\ref{E:sob-phi-norm-X}) and~(\ref{E:ell-2-1}), we get
\begin{align}\nonumber
&\|(G ((v_j)_{j=1}^{\infty})\|^2_{H^{\varphi}(X)}=\sum_{i=1}^{\infty}\|\left(h_i(G ((v_j)_{j=1}^{\infty}))\right)\circ \gamma_{e_i}\|^2_{H^{\varphi}(\RR^n)}\nonumber\\
&\leq \widetilde{C_6}\sum_{i=1}^{\infty}\sum_{j\in  \mathcal{A}(i)}\|v_j\|^2_{H^{\varphi}(\RR^n)}\leq \widetilde{C_7} \sum_{j=1}^{\infty}\|v_j\|^2_{H^{\varphi}(\RR^n)}=\widetilde{C_7}\|(v_j)_{j=1}^{\infty}\|^2_{\ell^2(H^{\varphi}(\RR^n))},\nonumber
\end{align}
where $\widetilde{C_7}$ is a constant.

Here, in the first inequality we used the fact that $\widetilde{C_6}$ is independent of $i$. In the second inequality we used uniform local finiteness of the cover $\{V_i=B(x_i,\varepsilon)\}$, as discussed in property (C1) of Section~\ref{SS-1-4}.

Therefore, $G\colon \ell^2(H^{\varphi}(\RR^n)) \to H^{\varphi}(X)$ is a bounded linear operator.
\end{proof}

\subsection{Key Proposition}
The following is a generalization of Proposition~\ref{L-3} to the (extended) Sobolev scale on a manifold of bounded geometry $X$:

\begin{prop}\label{L-5} Let $X$ be a manifold of bounded geometry. Let $s_0<s_1$ be two real numbers, let $\psi\in\mathcal{B}$ be an interpolation parameter, and let $\varphi$ be as in~(\ref{E:phi-psi}). Then, the following properties hold:
\begin{enumerate}
  \item [(i)] $\varphi\in RO$;
  \item [(ii)] up to norm equivalence we have
\begin{equation}\label{interp-r-n}
[H^{(s_0)}(X), H^{(s_1)}(X)]_{\psi}=H^{\varphi}(X).
\end{equation}
\end{enumerate}
\end{prop}

\begin{proof}
Let $s_0<s_1$ be real numbers, let $\psi\in\mathcal{B}$ be an interpolation parameter, and let $\varphi$ be as in~(\ref{E:phi-psi}).
Let $(V_j, \gamma_{e_j}, h_j)_{j=1}^{\infty}$ be a geodesic trivialization as in Section~\ref{SS-1-5}.

Keeping in mind the notation $\ell^2(\mathscr{H})$ from~(\ref{E:ell-2-1}), define a (linear) operator $F$ acting on $u\in H^{\varphi}(X)$ as follows:
\begin{equation}\label{E:op-f}
Fu:=((h_ju)\circ \gamma_{e_j})_{j=1}^{\infty}.
\end{equation}
Looking at~(\ref{E:sob-phi-norm-X}) and~(\ref{E:ell-2-1}), it is easily checked that
\begin{equation}\label{E:op-f-phi-fin}
F\colon H^{\varphi}(X)\to \ell^2(H^{\varphi}(\RR^n))
\end{equation}
is a bounded linear operator.

In particular, using~(\ref{E:op-f-phi-fin}) with $\varphi(t)=t^{s_0}$ and $\varphi(t)=t^{s_1}$, we get bounded linear operators
\begin{equation}\label{E:op-f-j}
F\colon H^{(s_0)}(X)\to \ell^2(H^{(s_0)}(\RR^n)),\qquad F\colon  H^{(s_1)}(X)\to \ell^2(H^{(s_1)}(\RR^n)).
\end{equation}
Since $\psi\in\mathcal{B}$ is an interpolation parameter, we also have the following bounded linear operator (see Section~\ref{SS-2-1}):
\begin{equation}\label{E:op-f-psi}
F\colon  [H^{(s_0)}(X), H^{(s_1)}(X)]_{\psi}\to [\ell^2(H^{(s_0)}(\RR^n)),\ell^2(H^{(s_1)}(\RR^n))]_{\psi}.
\end{equation}
By Lemma~\ref{L-6} and Proposition~\ref{L-3} we have
\begin{align}\label{E:ref-int}
&[\ell^2(H^{(s_0)}(\RR^n)),\ell^2(H^{(s_1)}(\RR^n))]_{\psi}=\ell^2([H^{(s_0)}(\RR^n),H^{(s_1)}(\RR^n)]_{\psi})\nonumber\\
&=\ell^2(H^{\varphi}(\RR^n)),
\end{align}
which together with~(\ref{E:op-f-psi}) leads to a bounded linear operator
\begin{equation}\label{E:op-f-psi-final}
F\colon [H^{(s_0)}(X), H^{(s_1)}(X)]_{\psi}\to \ell^2(H^{\varphi}(\RR^n)).
\end{equation}

Recalling the formulas for the actions of $G$ and $F$ in~(\ref{E:op-g}) and~(\ref{E:op-f}) respectively, a short calculation (see Remark 3.12 of~\cite{G-13}) shows that
\begin{equation}\label{E:inversion-gf}
GFu=u.
\end{equation}

As the operators in~(\ref{E:op-f-psi-final}) and~(\ref{E:op-g-bdd}) are bounded, we get a bounded linear operator
\begin{equation}\label{E:gf-phi-est}
GF\colon [H^{(s_0)}(X), H^{(s_1)}(X)]_{\psi} \to H^{\varphi}(X).
\end{equation}

The boundedness of the operator $GF$ in~(\ref{E:gf-phi-est}), together with the property~(\ref{E:inversion-gf}), leads to
\begin{equation*}
\|u\|_{H^{\varphi}(X)}=\|GF u\|_{H^{\varphi}(X)}\leq c_1\|u\|_{[H^{(s_0)}(X), H^{(s_1)}(X)]_{\psi}},
\end{equation*}
where $c_1>0$ is a constant.

By Lemma~\ref{L:G-bounded} with $\varphi(t)=t^{s_0}$ and $\varphi(t)=t^{s_1}$, we get bounded linear operators
\begin{equation}\nonumber
G\colon\ell^2(H^{(s_0)}(\RR^n)) \to H^{(s_0)}(X),\,\qquad G\colon\ell^2(H^{(s_1)}(\RR^n)) \to H^{(s_1)}(X),
\end{equation}
which leads to a bounded linear operator
\begin{equation}\nonumber
G\colon [\ell^2(H^{(s_0)}(\RR^n)),\ell^2(H^{(s_0)}(\RR^n))]_{\psi} \to [H^{(s_0)}(X), H^{(s_1)}(X)]_{\psi}.
\end{equation}
Referring again to~(\ref{E:ref-int}) we get a bounded linear operator
\begin{equation}\label{E:op-g-final}
G\colon\ell^2(H^{\varphi}(\RR^n)) \to [H^{(s_0)}(X), H^{(s_1)}(X)]_{\psi}.
\end{equation}

As the operators in~(\ref{E:op-g-final}) with~(\ref{E:op-f-phi-fin}) are bounded, we obtain a bounded linear operator
\begin{equation}\label{E:gf-phi-est-1}
GF\colon H^{\varphi}(X)\to [H^{(s_0)}(X), H^{(s_1)}(X)]_{\psi}.
\end{equation}
The boundedness of the operator $GF$ in~(\ref{E:gf-phi-est-1}), together with the property~(\ref{E:inversion-gf}), leads to
\begin{equation*}
\|u\|_{[H^{(s_0)}(X), H^{(s_1)}(X)]_{\psi}}=\|GF u\|_{[H^{(s_0)}(X), H^{(s_1)}(X)]_{\psi}}\leq c_2\|u\|_{H^{\varphi}(X)},
\end{equation*}
where $c_2>0$ is a constant.

Therefore, up to norm equivalence, we have the equality~(\ref{interp-r-n}).
\end{proof}

\subsection{Proof of Theorem~\ref{T:main-1}: Implication (i)$\implies$(ii)} As $\mathscr{S}$ is an interpolation space with respect to the pair $[H^{(s_0)}(X), H^{(s_1)}(X)]$, for some $s_0<s_1$, we can use Proposition~\ref{L-2} to infer $\mathscr{S}=\mathscr{H}_{\psi}$, where $\psi\in\mathcal{B}$ is pseudoconcave in a neighborhood of $\infty$. (In other words, in view of Proposition~\ref{L-1}, $\psi$ is an interpolation parameter.) Define $\varphi$ by the formula~(\ref{E:phi-psi}). By Proposition~\ref{L-3-b}, the function $\varphi$ satisfies the condition~(\ref{E:RO-3}), and, hence, $\varphi\in RO$.  To complete the proof, it remains to apply Proposition~\ref{L-5}. $\hfill\square$

\subsection{Proof of Theorem~\ref{T:main-1}: Implication (ii)$\implies$(i)} Let $\varphi\in RO$ be a function satisfying the condition~(\ref{E:RO-3}), and define $\psi$ as in~(\ref{E:phi-psi-inverse}), with $s_0<s_1$ as in~(\ref{E:RO-3}). By Proposition~\ref{L-4} it follows that $\psi\in\mathcal{B}$ and $\psi$ is an interpolation parameter. To complete the proof, it remains to use Proposition~\ref{L-5} with $s_0$ and $s_1$ as in~(\ref{E:phi-psi-inverse}). $\hfill\square$

\section{Proof of Corollary~\ref{C:cor-1}}\label{S:S-4}
The implication (i)$\implies$(ii) follows directly from Theorem~\ref{T:main-1}. As for the implication (ii)$\implies$(i), let $\varphi\in RO$ and let $s_0<\sigma_0(\varphi)$ and $s_1>\sigma_1(\varphi)$. Then $\varphi$ satisfies the condition~(\ref{E:RO-3}). Thus, by Theorem~\ref{T:main-1}, we have (up to norm equivalence) that $H^{\varphi}(X)$ is an interpolation space with respect to the pair
$[H^{(s_0)}(X), H^{(s_1)}(X)]$. Hence, up to the norm equivalence, $H^{\varphi}(X)$ is an interpolation space with respect to the scale $\{H^{(s)}(X)\colon s\in\RR\}$. $\hfill\square$

\section{Proof of Theorem~\ref{T:main-2}}\label{S:S-5}
By Proposition~\ref{L-4} we have $\psi\in\mathcal{B}$ and $\psi$ is an interpolation parameter. Furthermore, by the same proposition, $\psi$ satisfies~(\ref{E:phi-psi}). With these observations, Theorem~\ref{T:main-2} follows from Proposition~\ref{L-5}. $\hfill\square$

\section{Proof of Theorem~\ref{T:ind-triv}}\label{S:S-5a}
As in the proof of Theorem 3 in~\cite{MZ-24}, we consider two geodesic trivializations  $\mathscr{T}:=(V_j, \gamma_{e_j}, h_j)_{j=1}^{\infty}$ and $\widetilde{\mathscr{T}}:=(\widetilde{V_j}, \widetilde{\gamma_{e_j}}, \widetilde{h_j})_{j=1}^{\infty}$, with notations as in Section~\ref{SS-1-5}.

By Theorem 3.9 in~\cite{G-13} the identity map $I$ gives isomorphisms between the (usual) Sobolev spaces
\begin{equation}\label{E:iso-1-f}
I\colon H^{(s_0)}(X,\mathscr{T})\to H^{(s_0)}(X,\widetilde{\mathscr{T}}),\qquad I\colon H^{(s_1)}(X,\mathscr{T})\to H^{(s_1)}(X,\widetilde{\mathscr{T}}),
\end{equation}
where $(X,\mathscr{T})$ and $(X,\widetilde{\mathscr{T}})$ indicate that we used the geodesic trivializiations $\mathscr{T}$ and $\widetilde{\mathscr{T}}$ respectively.

Assume that $\varphi\in RO$ and choose numbers $s_0,\,s_1$ such that $s_0<\sigma_0(\varphi)$ and $s_1>\sigma_1(\varphi)$. Defining $\psi$ as in~(\ref{E:phi-psi-inverse}), we have by Theorem~\ref{T:main-2} that $\psi\in\mathcal{B}$ is an interpolation parameter and, up to norm equivalence,
\begin{equation}\label{E:iso-2-f}
[H^{(s_0)}(X,\mathscr{T}), H^{(s_1)}(X,\mathscr{T})]_{\psi}=H^{\varphi}(X,\mathscr{T}),
\end{equation}
\begin{equation}\label{E:iso-3-f}
[H^{(s_0)}(X,\widetilde{\mathscr{T}}), H^{(s_1)}(X,\widetilde{\mathscr{T}})]_{\psi}=H^{\varphi}(X,\widetilde{\mathscr{T}}),
\end{equation}

Putting together~(\ref{E:iso-1-f}),~(\ref{E:iso-2-f}),~(\ref{E:iso-3-f}) and using part (ii) of the definition of the interpolation space (see Section~\ref{SS-1-11}), it follows that the spaces $H^{\varphi}(X,\mathscr{T})$ and  $H^{\varphi}(X,\widetilde{\mathscr{T}})$ coincide up to norm equivalence. $\hfill\square$

\section{Proof of Theorem~\ref{T:main-3}}\label{S:S-6} Having equipped ourselves with Theorem~\ref{T:main-2}, we can just follow the proof of Theorem 5.2 in~\cite{MM-15} for the case $X=\Omega$, where $\Omega$ is a bounded domain with Lipschitz boundary.

We begin with an abstract proposition (see Theorem 1.3 in~\cite{MM-14}):

\begin{prop}\label{P-MAIN-3-1} Assume that $\lambda,\eta,\psi\in\mathcal{B}$ and that $\frac{\lambda}{\eta}$ is bounded in a neighborhood
of $\infty$. Let $\mathscr{H}$ be an admissible pair of Hilbert spaces. Then,

\begin{enumerate}
  \item [(i)] $[\mathscr{H}_{\lambda}, \mathscr{H}_{\eta}]$ is an admissible pair;
  \item [(ii)] we have $[\mathscr{H}_{\lambda}, \mathscr{H}_{\eta}]_{\psi}= \mathscr{H}_{\omega}$, with norm equality, where

\begin{equation}\label{E:eqn-def-greek}
\omega(t) := \lambda(t) \psi\left(\frac{\eta(t)}{\lambda(t)}\right),\qquad t>0.
\end{equation}
\item [(iii)] If $\lambda,\eta,\psi\in\mathcal{B}$ are interpolation parameters, then so is $\omega$.
\end{enumerate}
\end{prop}

To prove Theorem~\ref{T:main-3}, we start with $\varphi_0\in RO$ and choose numbers $s_0,\,s_1$ such that $s_0<\sigma_0(\varphi_j)$ and $s_1>\sigma_1(\varphi_j)$, $j=0,1$. With these $s_0$ and $s_1$, define $\lambda_j$,  $j=0,1$, as follows:
\begin{equation}\label{E:phi-psi-inverse-sigma}
\lambda_j(t):=\left\{\begin{array}{cc}
                  \tau^{-s_0/(s_1-s_0)}\varphi_j(\tau^{1/(s_1-s_0)}), & \tau\geq 1, \\
                  \varphi_j(1), & 0<\tau<1,
                \end{array}\right.
\end{equation}

Observe that $\frac{\lambda_0}{\lambda_1}$ is bounded in a neighborhood of $\infty$.

By Theorem~\ref{T:main-2} we have
\begin{equation}\label{E:temp-3-1}
[[H^{(s_0)}(X), H^{(s_1)}(X)]_{\lambda_0},[H^{(s_0)}(X), H^{(s_1)}(X)]_{\lambda_1}]=[H^{\varphi_0}(X),H^{\varphi_1}(X)].
\end{equation}
As the left hand side features an admissible pair (see part (i) of Proposition~\ref{P-MAIN-3-1}), the pair on the right hand side is also admissible. This proves part (i) of Theorem~\ref{T:main-3}.

Let $\psi$ be as in the hypothesis of Theorem~\ref{T:main-3}. Going back to~(\ref{E:temp-3-1}) and interpolating with a function parameter $\psi$, we get (after using part (ii) of Proposition~\ref{P-MAIN-3-1} with $\mathscr{H}=[H^{(s_0)}(X), H^{(s_1)}(X)]$)

\begin{equation}\label{E:temp-3-2}
[H^{(s_0)}(X), H^{(s_1)}(X)]_{\omega}=[H^{\varphi_0}(X),H^{\varphi_1}(X)]_{\psi},
\end{equation}
where (the interpolation parameter) $\omega$ is given by
\begin{equation}\label{E:phi-psi-inverse-sigma-1}
\omega(t):=\lambda_0(t) \psi\left(\frac{\lambda_1(t)}{\lambda_0(t)}\right),\qquad t\geq 1.
\end{equation}

As in the proof of Theorem 5.2 of~\cite{MM-15} it can be checked that the function $\varphi$ from the hypothesis~(\ref{E:quad-int}) satisfies the condition
\begin{equation*}
\varphi(t)=t^{s_0}\omega(t^{s_1-s_0}),\quad t\geq 1.
\end{equation*}

Therefore, by Proposition~\ref{L-5} we have $\varphi\in RO$  and
\begin{equation}\nonumber.
[H^{(s_0)}(X), H^{(s_1)}(X)]_{\omega}=H^{\varphi}(X),
\end{equation}
and, keeping in mind~(\ref{E:temp-3-2}), we obtain~(\ref{E:t-3-main}). $\hfill\square$

\section{Proof of Theorem~\ref{T:main-4}}\label{S:S-7}
With Theorem~\ref{T:main-2} at our disposal, the proofs of the properties (i)--(iii) proceed in the same manner as in~\cite{MM-13} for the case $X=\RR^n$  (or as in~\cite{MZ-24} for the case of extended Sobolev scale on a closed manifold $X$.) 

For part (i), start with $\varphi\in RO$ and choose numbers $s_0,\,s_1$ such that $s_0<\sigma_0(\varphi)$ and $s_1>\sigma_1(\varphi)$. Defining $\psi$ as in~(\ref{E:phi-psi-inverse}), we have by Theorem~\ref{T:main-2} that $\psi\in\mathcal{B}$ is an interpolation parameter and, up to norm equivalence, we have
\begin{equation*}
[H^{(s_0)}(X), H^{(s_1)}(X)]_{\psi}=H^{\varphi}(X).
\end{equation*}
Appealing to part (i) of the definition of the interpolation space (see Section~\ref{SS-1-11}), we get a continuous embedding
\begin{equation*}
H^{(s_1)}(X)\hookrightarrow H^{\varphi}(X)\hookrightarrow H^{(s_0)}(X)
\end{equation*}
and the density of $\xcomp$ in $H^{\varphi}(X)$ follows from the fact (see Proposition 3.11 in~\cite{Kor-91}) that $\xcomp$ is dense in the spaces $H^{(s)}(X)$, $s\in\RR$.

For part (ii), it suffices to observe that hypotheses of Theorem~\ref{T:main-3} are satisfied; therefore,
$[H^{\varphi_0}(X),H^{\varphi_1}(X)]$ is an admissible pair, and we have a continuous embedding $H^{\varphi_1}(X)\hookrightarrow H^{\varphi_0}(X)$.

For part (iii), we proceed as in the proof of Theorem 7 in~\cite{MZ-24}. First, recall (see Proposition 3.13 in~\cite{Kor-91}) that for all $s\in\RR$ the sesquilinear form~(\ref{E:inner-l-2-x}) extends to a sesquilinear duality (separately continuous sesquilinear form)
\begin{equation}\label{E:inner-product-h-phi-reg}
(\cdot,\cdot)\colon H^{(s)}(X)\times H^{(-s)}(X)\to\mathbb{C}.
\end{equation}
The spaces $H^{(s)}(X)$ and $H^{(-s)}(X)$ are dual relative to the duality~(\ref{E:inner-product-h-phi-reg}): for each $s\in \RR$, the map $f(u):=(u,\cdot)$, with $u\in H^{(s)}(X)$, is an isomorphism $f\colon H^{(s)}(X)\to (H^{(-s)}(X))'$, where $[H^{(-s)}(X)]'$ is the anti-dual space of $H^{(-s)}(X)$.

Let $\varphi\in RO$, let $s_0<\sigma_0(\varphi)$, $s_1>\sigma_1(\varphi)$, and let $\psi$ be as in~(\ref{E:phi-psi-inverse}). Thus, we have an isomorphism
\begin{equation}\label{E:4-1-in}
[H^{(s_0)}(X),H^{(s_1)}(X)]_{\psi}\cong [(H^{(-s_0)}(X))',(H^{(-s_1)}(X))']_{\psi}.
\end{equation}
Furthermore, by abstract Theorem 1.4 in~\cite{MM-14}, we have (with equality of norms)
\begin{equation}\label{E:4-2-in}
[(H^{(-s_0)}(X))',(H^{(-s_1)}(X))']_{\psi}=\left([(H^{(-s_1)}(X)),(H^{(-s_0)}(X))]_{\widetilde{\psi}}\right)',
\end{equation}
where $\widetilde{\psi}:=\frac{t}{\psi(t)}$. Referring again to Theorem 1.4 from~\cite{MM-14}, $\widetilde{\psi}\in\mathcal{B}$ is an interpolation parameter.

As $\psi$ satisfies~(\ref{E:phi-psi}) (as guaranteed by Proposition~\ref{L-4}), it is easily checked that $\widetilde{\varphi}:=\frac{1}{\varphi}$ satisfies
\begin{equation}\nonumber
\widetilde{\varphi}(t)=t^{-s_1}\widetilde{\psi}(t^{-s_0-(-s_1)}).
\end{equation}
Therefore, by Proposition~\ref{L-5}, we have, up to norm equivalence,
\begin{equation}\label{E:4-3-in}
[(H^{(-s_1)}(X)),(H^{(-s_0)}(X))]_{\widetilde{\psi}}=H^{\frac{1}{\varphi}}(X).
\end{equation}
Combining~(\ref{E:4-1-in}),~(\ref{E:4-2-in}),~(\ref{E:4-3-in}) and recalling (see Theorem~\ref{T:main-2}) that
\begin{equation}\nonumber
H^{\varphi}(X)=[H^{(s_0)}(X),H^{(s_1)}(X)]_{\psi},
\end{equation}
we get an isomorphism
\begin{equation}\nonumber
H^{\varphi}(X)\cong  \left(H^{\frac{1}{\varphi}}(X)\right)'.
\end{equation}
Thus, we have mutual duality of the spaces  $H^{\varphi}(X)$ and $H^{\frac{1}{\varphi}}(X)$ with respect to
the sesquilinear form~(\ref{E:inner-product-h-phi}).

Since~(\ref{E:inner-product-h-phi-reg}) is an extension by continuity of the form~(\ref{E:inner-l-2-x}) and since we have continuous embeddings $H^{(s_1)}(X)\hookrightarrow H^{\varphi}(X)$ and $H^{(-s_0)}(X) \hookrightarrow H^{\frac{1}{\varphi}}(X)$,  the form~(\ref{E:inner-product-h-phi}) is an extension by continuity of the form~(\ref{E:inner-l-2-x}). $\hfill\square$

\section{Proof of Theorem~\ref{T:main-5}}\label{S:S-8}
Let $\varphi\in RO$ and let $s_0<\sigma_0(\varphi)$ and $s_1>\sigma_1(\varphi)$. For $A\in \uprl^{m}(X)$, $m\in\RR$, Theorem 3.9 of~\cite{Kor-91} tells us that
\begin{equation}\label{E:bdd-A-kordyukov}
A\colon H^{(s_0)}(X)\to H^{(s_0-m)}(X),\qquad A\colon H^{(s_1)}(X)\to H^{(s_1-m)}(X)
\end{equation}
are bounded linear operators.

Defining $\psi$ as in~(\ref{E:phi-psi-inverse}) and using Proposition~\ref{L-4} we infer that $\psi\in\mathcal{B}$ is an interpolation parameter satisfying~(\ref{E:phi-psi}).  Furthermore, by Proposition~\ref{L-5} we have (up to norm equivalence)
\begin{equation}\label{E:interp-A-1}
[H^{(s_0)}(X), H^{(s_1)}(X)]_{\psi}=H^{\varphi}(X).
\end{equation}

On the other hand, defining $\widetilde{\varphi}:=t^{-m}\varphi$, we see that~(\ref{E:phi-psi}) can be written as
\begin{equation}\nonumber
\varphi(t)=t^{s_0-m}t^{m}\psi(t^{(s_1-m)-(s_0-m)}),
\end{equation}
that is,
\begin{equation}\nonumber
\widetilde{\varphi}(t)=t^{s_0-m}\psi(t^{(s_1-m)-(s_0-m)}),
\end{equation}

Thus, we can use Proposition~\ref{L-5} with $s_0-m$ and $s_1-m$ instead of $s_0$ and $s_1$ respectively and with $\widetilde{\varphi}=t^{-m}\varphi$ instead of $\varphi$. This leads to (up to norm equivalence)
\begin{equation}\label{E:interp-A-2}
[H^{(s_0-m)}(X), H^{(s_1-m)}(X)]_{\psi}=H^{t^{-m}\varphi}(X).
\end{equation}

Taking into account bounded linear operators~(\ref{E:bdd-A-kordyukov}), keeping in mind the properties~(\ref{E:interp-A-1})--(\ref{E:interp-A-2}), and recalling the definition of interpolation space (see Section~\ref{SS-1-11}), we
get a bounded linear operator $A\colon H^{\varphi}(X)\to H^{t^{-m}\varphi}(X)$.

Similarly, if $\mathscr{F}$ is a bounded family in $\uprl^{m}(X)$, then there exist $C_{s_0}$ and $C_{s_1}$ such that
\begin{equation}\label{E:bdd-A-kordyukov-fm}
\|A\|_{H^{(s_0)}(X)\to H^{(s_0-m)}(X)}\leq C_{s_0},\qquad \|A\|_{H^{(s_1)}(X)\to H^{(s_1-m)}(X)}\leq C_{s_1},
\end{equation}
for all $A\in \mathscr{F}$. (Here, the constants are independent of $A\in \mathscr{F}$.)

Keeping in mind the definition of interpolation space (see Section~\ref{SS-1-11}) and Remark~\ref{R-rem-int}, the estimate~(\ref{E:bdd-A-fm}) follows by the same interpolation argument as in the first part of the proof.
$\hfill\square$

\section{Proof of Theorem~\ref{T:main-6}}\label{S:S-9}
We proceed as in Theorem 3.12 of~\cite{Kor-91} for the usual Sobolev scale $H^{(s)}(X)$ on a manifold of bounded geometry $X$. We will prove the result for $k=0$, and outline the proof for $k=1,2,\dots$.

Let $(V_j, \gamma_{e_j}, h_j)_{j=1}^{\infty}$ be a geodesic trivialization as in Section~\ref{SS-1-5}. Let $\eta\in C^{\infty}_{c}(\RR^n)$ be a function such that $\eta(t)=1$ near $0\in \RR^n$ and $\supp\eta\subset \widehat{K}_{\varepsilon}$, where $0<\varepsilon<r_{inj}$ and $\widehat{K}_{\varepsilon}$ is as in Section~\ref{SS-1-2}.

For $x_0\in X$ define a family of functions $\chi_{x_0}\colon X\to \RR$, indexed by $x_0\in X$, as follows: $\chi_{x_0}:=\eta\circ \kappa_{e_{x_0}}$, where $\kappa_{e_{x_0}}$ is as in~(\ref{E:geo-def-kappa}), $e_{x_0}\in O(X)$ and $\pi(e_{x_0})=x_0$.

With these notations we have:

\begin{align}\label{E:Eqn-1}
&\sup_{x\in X}|u(x)|=\sup_{x,x_0\in X}|\chi_{x_0}(x)u(x)|\nonumber\\
&\leq\sup_{x_0\in X}\left(\sup_{x\in X}\sum_{j=1}^{\infty}|\chi_{x_0}(x)h_{j}(x)u(x)|\right)\nonumber\\
&=\sup_{x_0\in X}\left(\sup_{p\in {\RR}^n}\sum_{j=1}^{\infty}|[(\chi_{x_0}h_{j}u)\circ \gamma_{e_{j}}](p)|\right)\nonumber\\
&\leq \widetilde{C_1}\sup_{x_0\in X}\left(\sum_{j=1}^{\infty}\|[\chi_{x_0}h_{j}u]\circ\gamma_{e_j}\|_{H^{\varphi}(\RR^n)}\right)\nonumber\\
&\leq\widetilde{C_2}\sup_{x_0\in X}\|\chi_{x_0}u\|_{H^{\varphi}(X)},
\end{align}
where $\widetilde{C_1}$ and $\widetilde{C_2}$ are constants.

Here, the first inequality follows from the properties of $h_j$ (see Section~\ref{SS-1-4}). The second inequality is obtained by applying  the continuous embedding $H^{\varphi}(\RR^n)\hookrightarrow C_{b}(\RR^n)$ (which holds under the hypothesis~(\ref{E:emb-cb-k}); see part (iv) of Proposition 2 in~\cite{MM-13}). Thanks to the properties of the functions $h_j$ (see Section~\ref{SS-1-4}), the constant $\widetilde{C_1}$ in the second inequality is independent of the index $j$. The third inequality follows from the definition~(\ref{E:sob-phi-norm-X}).

Furthermore, since the multiplication operators by a function $\chi_{x_0}$  form a bounded family $\{\chi_{x_0}\}_{x_0\in X}$ of operators in $\uprl^{0}(X)$ (here, the term ``bounded family" is as in Definition~\ref{D-6-x-bdd}), we can use Theorem~\ref{T:main-5} to get
\begin{equation}\label{E:Eqn-2}
\widetilde{C_2}\sup_{x_0\in X}\|\chi_{x_0}u\|_{H^{\varphi}(X)}\leq \widetilde{C_3}\|u\|_{H^{\varphi}(X)},
\end{equation}
where $\widetilde{C_3}$ is a constant.

Combining~(\ref{E:Eqn-1}) and~(\ref{E:Eqn-2}) completes the proof for the case $k=0$.

The general case can be shown by induction with the help of the following observations: First, looking at the definitions~(\ref{E:FT}) and~(\ref{E:sob-phi-1}),  we note that
\begin{equation}\label{E:induction-1}
\|(h_j\circ\gamma_{e_j})[\partial_{i}(u\circ\gamma_{e_j})]\|_{H^{\widetilde{\varphi}}(\RR^n)}\leq \widetilde{C} \|(h_ju)\circ\gamma_{e_j}\|_{H^{\varphi}(\RR^n)},
\end{equation}
where $\widetilde{\varphi}(t):=t^{-1}\varphi(t)$, $\partial_{i}:=\frac{\partial}{\partial_{x_i}}$, and $\widetilde{C}$ is a constant.


Furthermore, note that if $\varphi(t)$ satisfies the condition
\begin{equation}\nonumber
\int_{1}^{\infty}\frac{t^{2(k+1)+n-1}}{\varphi^2(t)}\,dt<\infty,
\end{equation}
then $\widetilde{\varphi}(t):=t^{-1}\varphi(t)$ satisfies the condition
\begin{equation}\nonumber
\int_{1}^{\infty}\frac{t^{2k+n-1}}{[\widetilde\varphi(t)]^2}\,dt<\infty.
\end{equation}
This concludes the proof. $\hfill\square$

\section{Proof of Theorem~\ref{T:main-7}}\label{S:S-10} We begin by recalling an abstract result (see Theorem 2.5 in~\cite{MM-21}):

\begin{prop}\label{P:6-1}Let $\varphi\in RO$, $s_0<\sigma_0(\varphi)$, and $s_1>\sigma_1(\varphi)$. Let $\psi$ be as in~(\ref{E:phi-psi-inverse}).

Then, $\psi\in\mathcal{B}$ and $\psi$ is an interpolation parameter. Furthermore, we have
\begin{equation*}
[H^{(s_0)}_{A}, H^{(s_1)}_{A}]_{\psi}=H^{\varphi}_{A},
\end{equation*}
with equality of norms, where $H^{(s_j)}_{A}$, $j=0,1$, and $H^{\varphi}_{A}$ are as in Section~\ref{SS-1-21}.
\end{prop}

We now return to the setting of Theorem~\ref{T:main-7}.

\begin{lemma}\label{L:A-k-isomorphic} Assume that $X$ is a manifold of bounded geometry. Assume that $A$ satisfies the hypotheses (H1)--(H2). Then, up to norm equivalence, we have
\begin{equation*}
H_{A}^{(k)}(X)=H^{(k)}(X),\quad \textrm{for all }k\in\mathbb{Z}.
\end{equation*}
\end{lemma}
\begin{proof}
It is enough to show the result for $k\in\NN_{0}$, as the case  $k=-1,-2,\dots$ follows by duality from the case $k\in\NN$.
The case $k=0$ is obviously true. Thus, it remains to consider the case $k\in\NN$.

First, recalling (H1)--(H2) and looking at Remark~\ref{R:kor}, we see that $A$ is a positive self-adjoint operator in $L^2(X)$ such that $\dom(A)=H^{(1)}(X)$. Therefore $A$ establishes an isomorphism $A\colon H^{(1)}(X)\to L^2(X)$. This, together with the definition of $H^{(1)}_{A}(X)$, tells us that $H^{(1)}_{A}(X)=H^{(1)}(X)$, up to norm equivalence.

Furthermore, we recall the following consequence of ``elliptic regularity" property (see Theorem 3.6 in~\cite{Kor-91}) for our operator $A\in \ueprl^{1}(X)$: if $s>1$ and if $u\in H^{(1)}(X)$ satisfies $Au\in H^{(s-1)}(X)$, then $u\in H^{(s)}(X)$.

Using this observation and the mentioned isomorphism $A\colon H^{(1)}(X)\to L^2(X)$, we see that $A^k$ gives rise to an isomorphism $A^k\colon H^{(k)}(X)\to L^2(X)$. Combining this property with the definition of $H^{(k)}_{A}(X)$, we infer that $H^{(k)}_{A}(X)=H^{(k)}(X)$, up to norm equivalence.
\end{proof}

\noindent\textbf{Continuation of the Proof of Theorem~\ref{T:main-7}}
\\\\
Let $\varphi\in RO$, let $k\in\NN$ be a number such that $-k<\sigma_0(\varphi)$ and $k>\sigma_1(\varphi)$. Define $\psi$ as in~(\ref{E:phi-psi-inverse}) with $s_0=-k$ and $s_1=k$. By Proposition~\ref{P:6-1} we have, up to norm equivalence,
\begin{equation}\label{E:m-7-1}
[H^{(-k)}_{A}, H^{(k)}_{A}(X)]_{\psi}=H^{\varphi}_{A}(X).
\end{equation}

The property~(\ref{E:m-7-1}) in combination with Lemma~\ref{L:A-k-isomorphic} yields, up to norm equivalence,
\begin{equation*}
[H^{(-k)}(X), H^{(k)}(X)]_{\psi}=H^{\varphi}_{A}(X).
\end{equation*}
On the other hand, Theorem~\ref{T:main-2} tells us (up to norm equivalence)
\begin{equation*}
H^{\varphi}(X)=[H^{(-k)}(X), H^{(k)}(X)]_{\psi}.
\end{equation*}
Combining the last two properties, we obtain, up to norm equivalence, $H^{\varphi}_{A}(X)=H^{\varphi}(X)$. $\hfill\square$

\end{document}